\documentclass[12pt,reqno]{amsart}
\usepackage{amssymb}
\usepackage{amsmath}
\usepackage{color}
\usepackage{graphics, cite, bookmark}
\usepackage{epsfig}
\usepackage{hyperref}
\usepackage{epstopdf}
\newcommand{\R}{\mathbb{R}}
\newcommand{\N}{\mathbb{N}}

\newcommand{\eps}{ \varepsilon}

\newcommand{\tr}{\mathrm{tr}}
\newcommand{\C}{\mathbb{C}}
\newcommand{\supp}{\textrm{supp }}

\numberwithin{equation}{section}

\newtheorem{theorem}{Theorem}[section]
\newtheorem*{theorem*}{Theorem}
\newtheorem{corollary}{Corollary}[section]
\newtheorem{proposition}{Proposition}[section]
\newtheorem{lemma}{Lemma}[section]

\theoremstyle{definition}
\newtheorem{definition}{Definition}[section]

\newtheorem{condition}{Condition}[section]

\author{Wentao Cao}
\address{Institut f\"{u}r mathematik, Universit\"{a}t Leipzig, Augustusplatz 10, D-04109, Leipzig, Germany}
\email{wentao.cao@math.uni-leipzig.de}
\author{L\'aszl\'o Sz\'ekelyhidi Jr.}
\address{Institut f\"{u}r mathematik, Universit\"{a}t Leipzig, Augustusplatz 10, D-04109, Leipzig, Germany}
\email{laszlo.szekelyhidi@math.uni-leipzig.de}

\title[Global Nash-Kuiper theorem]
{Global Nash-Kuiper theorem for compact manifolds}

\date{\today}


\begin{document}

\begin{abstract}
We obtain global extensions of the celebrated Nash-Kuiper theorem for $C^{1,\theta}$ isometric immersions of compact manifolds with optimal H\"older exponent. In particular for the Weyl problem of isometrically embedding a convex compact surface in 3-space, we show that the Nash-Kuiper non-rigidity prevails upto exponent $\theta<1/5$. This extends previous results on embedding 2-discs as well as higher dimensional analogues.\end{abstract}

\maketitle

%
\section{Introduction}
 Let $(\mathcal{M}, g)$ be a compact $n$-dimensional manifold with $C^1$ metric $g$.
 The celebrated Nash-Kuiper theorem \cite{Nash54,Kui55} states that any short immersion or embedding $u:\mathcal{M}\hookrightarrow \R^{n+1}$ can be uniformly approximated by $C^1$ isometric immersions/embeddings. As a particular case, for the classical Weyl problem, i.e.~$(S^2,g)\hookrightarrow \R^3$ with positive Gauss curvature $\kappa_g>0$ this result implies the existence of a vast set of non-congruent $C^1$ surfaces, each isometric to $(S^2,g)$. This is in stark contrast with the situation for $C^2$ isometric embeddings: the famous rigidity theorem of Cohn-Vossen \cite{CohnVossen} and Herglotz \cite{Herglotz:1943je} states that the $C^2$ isometric embedding $(S^2,g)\hookrightarrow \R^3$ is uniquely determined up to congruencies. 
 
 The question of what happens inbetween the rigid $C^2$ case and the highly non-rigid $C^1$ case on the H\"older scale $C^{1,\theta}$ has a long history. In the 1950s in a series of papers \cite{Borisov58,BorisovRigidity1} Borisov, building upon the work of Pogorelov \cite{PogorelovRigidity}, showed that the rigidity of convex surfaces prevails for $\theta>2/3$. A short modern proof based on regularization and a commutator estimate was provided in \cite{CDS12}. More recently there has been intensive work on lowering the rigidity exponent \cite{DeLellis:2018ub,Gladbach:2019ti}, one conjecture being that some form of rigidity should hold for all $\theta>1/2$ \cite{Gromov:2015tua,DeLellis:2017dt,DeLellis:2018ub}.
 
Regarding the flexible side, Borisov announced in \cite{Bor65} that the Nash-Kuiper statement continues to be valid locally (i.e.~for embedding Euclidean balls) for any $\theta<(1+n(n+1))^{-1}$ and globally (i.e.~embedding general compact manifolds) for any $\theta<(1+n(n+1)^2)^{-1}$. In particular for the Weyl problem this yields $\theta<1/7$. A detailed proof for the 2-dimensional case for $\theta<1/13$ appeared in \cite{Bor04}. Subsequently, a simplified proof for the general case appeared in \cite{CDS12}. More recently, in \cite{DIS15} the exponent was raised to $\theta<1/5$ for the special case of embedding 2-dimensional discs $(D^2,g)\hookrightarrow\R^3$ by employing reparametrisations to conformal coordinates. Although not explicitly stated in the paper \cite{DIS15}, the same technique would be extendible to $(S^2,g)\hookrightarrow\R^3$ with $\theta<1/7$. 

One of the key features in all previous works is that global embeddings suffer from a loss of regularity and lower H\"older exponents compared to local results. In the present paper our main goal is to obtain global results with the optimal exponents, thus rectifying this shortcoming. 

\begin{theorem}[Main Theorem - 2D case]\label{t:global}
Let $(\mathcal{M},g)$ be a 2-dimensional compact manifold with $C^1$ metric and let $u:(\mathcal{M}, g)\hookrightarrow \R^3$ be a short immersion of class $C^1$. For any $\eps>0$ and any $0<\theta<{1}/{5}$ there exists a $C^{1, \theta}$ isometric immersion $v:(\mathcal{M}, g)\hookrightarrow \R^3$ such that  $\|v-u\|_{C^0(\mathcal{M})}<\eps$. Moreover, if $u$ is an embedding, $v$ can be chosen to be an embedding.
\end{theorem}

More generally, we are able to obtain a similar result for general $n$-dimensional case, although, since conformal coordinates are not available in this case, the H\"older exponent for $n=2$ is worse than in Theorem \ref{t:global} above.

\begin{theorem}[Main Theorem - nD case]\label{t:global-n}
Let $(\mathcal{M},g)$ be an n-dimensional compact manifold, $n\geq 3$, with $C^1$ metric and let $u:(\mathcal{M}, g)\hookrightarrow \R^{n+1}$ be a short immersion of class $C^1$. For any $\eps>0$ and any $0<\theta<\frac{1}{n^2+n+1}$ there exists a $C^{1, \theta}$ isometric immersion $v:(\mathcal{M}, g)\hookrightarrow \R^{n+1}$ such that  $\|v-u\|_{C^0(\mathcal{M})}<\eps$. Moreover, if $u$ is an embedding, $v$ can be chosen to be an embedding.
\end{theorem}

The main technique for showing such theorems is the iteration scheme known as convex integration. Originally introduced in \cite{Nash54}, convex integration became a widely used and very powerful technique through the seminal works of Gromov \cite{Gr73, Gr86}  for dealing with various (often underdetermined) systems of partial differential equations arising in geometry and topology, especially in association with the h-principle. Whilst Gromov's general framework applies very well to open or ample relations, the isometric embedding problem, written in local coordinates as
$\partial_iu\cdot\partial_ju=g_{ij}$ amounts to a first order partial differential relation where the analytic convergence (i.e.~local) aspects have to be addressed in detail. 

Interest in this problem was to a large extent revived in the last decade by the discovery of strong connections to fluid mechanics and the question of anomalous dissipation in ideal turbulence \cite{DS09,DS13, DS14, BDIS15}. Indeed, by now the parallel story of optimal H\"older exponents for the Euler equations in fluid mechanics is much more complete, Onsager's conjecture, that is, reaching the optimal exponent $1/3$, has recently  been  proved in \cite{Ise16} (see also \cite{BDSV17}). Motivated by these successes, convex integration has become a widely used tool in the PDE community, for example  compressible Euler systems in \cite{CDK15}, Monge-Ampere equation in \cite{LP17}, active scalar equations in \cite{IV15}, Navier-Stokes equations \cite{Buckmaster:2017wfa}. We also refer to the survey \cite{DeLellis:2017dt}. 

Returning to the isometric immersion problem, the key issue leading to a loss of regularity in global results is that Nash's iteration amounts to successively adding primitive metrics which requires working in local coordinates. To extend this successive improvement to a global setting then requires ``gluing'' the various coordinate patches. We are able to handle this problem by working on a fixed triangulation of the manifold and performing the gluing by an induction on the dimension of the skeleta. As such, our technique involves a special class of short immersions, called adapted short immersions, which were originally introduced for the Cauchy problem of the Euler equations \cite{D14,DS17} and used for the isometric extension problem in \cite{HW17} and \cite{CS19}. In a nutshell, adapted short immersions link in a quantitative way the deviation from being isometric to the blow-up of the second derivative. 

Our key construction is Proposition \ref{p:inductive}, which amounts to a kind of sharp quantitative version of what is known as the ``relative h-principle'' - in a sense the associated boundary value problem for dealing with short immersions which are already isometric on a (lower-dimensional) compact set. Our technique is very flexible in the sense that we are able to transfer any local Nash-Kuiper statement to the global setting. Indeed, we expect that any further improvement of the local $1/5$ exponent will be transferable to general compact surfaces by a variant of our Proposition \ref{p:inductive}. 

The paper is organised as follows. In Section \ref{s;preliminary} we fix the notation, recall some useful propositions and lemmata, and introduce adapted short immersions. In Section \ref{s:iterate} we provide versions of the Nash-Kuiper technique of adding primitive metrics - the main difference to previous versions of these results (e.g.~\cite{CDS12,DIS15,CS19}) is that we need to localize and consider compactly supported perturbations - this requires using a different mollification scale. Our main inductive proposition, essentially the main new technical element in this paper, is contained in Section \ref{s:inductive}. Here we focus on the case $n=2$ (Sections \ref{ss:n=2}-\ref{ss:conlude}), since in light of the Weyl problem and rigidity this is the most interesting (and, due to the conformal transformation, technically most involved) case, and we merely indicate in Section \ref{ss:n=3} the minor changes required for the case of general dimension. Finally, in Section \ref{s:proof-of-theorem} we show how to perform the induction over skeleta of a triangulation in our quantitative $C^{1,\theta}$ setting.

\section*{Acknowledgments}
Part of the work was completed when the authors were visiting Hausdorff Research Institute for Mathematics (HIM), they would like to express their gratitude for the hospitality and the wonderful working environment of HIM. 
The authors also gratefully acknowledge the support of the ERC Grant Agreement No. 724298.

\section{Preliminaries}\label{s;preliminary}

\subsection{Notation}

Throughout the paper we consider a compact Riemannian manifold $(\mathcal{M}, g)$ with a $C^1$ metric $g$. We fix a finite atlas $\{\Omega_k\}_k$ of $\mathcal{M}$ with charts $\Omega_k$ and a corresponding partition of unity $\{\phi_k\}$ so that
$$
\sum \phi_k^2=1 \text{ and }\phi_k\in C_c^{\infty}(\Omega_k).
$$
Furthermore, on each $\Omega_k$ we fix a choice of coordinates and in this way identify $\Omega_k$ with a bounded open subset of $\R^2$. We write $g=(G_{ij})$ in any local chart $\Omega_k$ and the same way denote any symmetric 2-tensor $h$ in local charts by $(H_{ij})$. Observe that, since $g$ is a non-degenerate metric and $\mathcal{M}$ is compact, there exists a constant $\gamma_0\geq 1$ such that, in any local chart $\Omega_k$, we have
\begin{equation}\label{e:gamma0}
\frac{1}{\gamma_0}\textrm{Id}\leq G\leq\gamma_0\textrm{Id}.
\end{equation}
Here and in the sequel an inequality of the type $G\leq H$ for symmetric 2-tensors means that $H-G$ is positive definite; equivalently, $G_{ij}\xi_i\xi_j\leq H_{ij}\xi_i\xi_j$ for all $\xi\in \R^2$.
In the paper $M, C$ denote constants independent of any other parameters, while $C(\cdot), C_j(\cdot), j\in\mathbb{N},$ denotes constants depending upon the parameters in the bracket.

As usual, we define the supremum norm of maps $f:\mathcal{M}\to\R^n$ as $\|f\|_0=\sup_{x\in\mathcal{M}}|f(x)|$. The H\"older seminorms on $\R^n$ are defined as
\begin{equation*}
[f]_{m}=\max_{|\beta|=m}\|D^{\beta}f\|_0\,,\quad
[f]_{m+\alpha} = \max_{|\beta|=m}\sup_{x\neq y}\frac{|D^{\beta}f(x)-D^{\beta}f(y)|}{|x-y|^{\alpha}}\, ,
\end{equation*}
where $\beta$ is a multiindex. The H\"older norms are then given by
\begin{equation*}
\|f\|_{m}=\sum_{j=0}^m[f]_j\,,\quad \|f\|_{m+\alpha}=\|f\|_m+[f]_{m+\alpha}.
\end{equation*}
We recall the following interpolation inequalities for these norms:
\begin{equation}\label{e:inerpolation}
\|f\|_{k,\alpha}\leq C\|f\|_{k_1,\alpha_1}^{\lambda}\|f\|_{k_2,\alpha_2}^{1-\lambda},
\end{equation}
where $C$ depends on the various parameters, $0<\lambda<1$ and
$$
k+\alpha=\lambda(k_1+\alpha_1)+(1-\lambda)(k_2+\alpha_2).
$$
We have the following standard estimates on mollifications of H\"older functions (see for instance \cite{CDS12} Lemma 1 for a proof).
\begin{proposition}\label{p:mollification}
Let $\varphi\in C_c^\infty(B_1(0))$ be symmetric, nonnegative and $\int\varphi=1$. For any $r, s\geq0,$ and $0<\alpha\leq1,$ we have for
\begin{itemize}
\item[(1)]$\|f*\varphi_\ell\|_{r+s}\leq C(r, s)\ell^{-s}\|f\|_r$,
\item[(2)]If $0\leq r\leq1,$ $\|f-f*\varphi_\ell\|_r\leq C(r)\ell^{1-r}\|f\|_1$,
\item[(3)] $\|(f_1f_2)*\varphi_\ell-(f_1*\varphi_\ell)(f_2*\varphi_\ell)\|_r\leq C(r, \alpha)\ell^{2\alpha-r}\|f_1\|_\alpha\|f_2\|_\alpha.$
\end{itemize}
\end{proposition}
Finally, for functions and tensors on $\mathcal{M}$ we use the given atlas and associated partition of unity to define the H\"older norms: for any $r\geq 0$ we set
$$
[u]_r:=\sum_k[\phi_k^2u_k]_r.
$$
Similarly, we define ``mollification on $\mathcal{M}$"  through the partition of unity. That is to say, for a function $u$ on $\mathcal{M}$ we define
$$
u*\varphi_\ell=\sum_k (u_k\phi_k^2)*\varphi_\ell.
$$
One can also check that the estimates in Proposition \ref{p:mollification} still hold on $\mathcal{M}$ with these definitions. Other properties about H\"older norm can be found in many references such as \cite{DIS15}. Finally, we note that these definitions can be easily extended to symmetric 2-tensors $h$ on $\mathcal{M}$ using the pointwise norm given by the underlying metric $g$:
$$
|h(x)|=\sup_{\xi\in T_x\mathcal{M}, |\xi|_g=1}|h(\xi,\xi)|,
$$
where $|\xi|_g=(\sum_{ij}g_{ij}\xi_i\xi_j)^{1/2}$.
Note that because of \eqref{e:gamma0} this norm is equivalent to the matrix norm of $H(x)$ given by
$$
|H(x)|=\sup_{|\xi|=1}|H_{ij}(x)\xi_i\xi_j|
$$
In particular, given the $C^1$ metric $g$ on $\mathcal{M}$ (with the local representation $g=(G_{ij})$ in local charts $\Omega_k$), we may choose $\gamma_0$ from \eqref{e:gamma0} sufficiently large so that in addition
\begin{equation}\label{e:gamma1}
	\|g\|_1\leq \gamma_0.
\end{equation}

\subsection{Corrugation}

Next, we recall the corrugation functions used in the Nash-Kuiper iteration \cite{Kui55,CDS12}, which we will require in the quantitative form used in Proposition 2.4 in \cite{DIS15}:
\begin{lemma}\label{l:gamma}
There exists $\delta_*>0$ and a smooth function $\Gamma=(\Gamma_1,\Gamma_2)(s, t)$ defined on $[0, \delta_*]\times\R$ satisfying the following properties
\begin{item}
\item[(1)]$\Gamma(s, t)=\Gamma(s, t+2\pi)$ for any $(s, t);$
\item[(2)]$(1+\partial_t\Gamma_1)^2+(\partial_t\Gamma_2)^2=1+s^2;$
\item[(3)] Three estimates about the derivatives of $\Gamma$ hold: for any $k\in\mathbb{N},$
\begin{equation}\label{e:gammaineq}
\begin{split}
|\partial_t^k\Gamma_1(s, \cdot)|\leq C(k)s^2,~~&
|\partial_t^k\Gamma_2(s, \cdot)|\leq C(k)s;\\
|\partial_s\partial_t^k\Gamma_1(s, \cdot)|\leq C(k)s,~~&
|\partial_s\partial_t^k\Gamma_2(s, \cdot)|\leq C(k);\\
|\partial_s^2\Gamma_1(s, \cdot)|\leq C, ~~&|\partial_s^2\Gamma_2(s, \cdot)|\leq C.
\end{split}
\end{equation}
\end{item}
\end{lemma}

\subsection{Decomposition into primitive metrics}

We recall the following decomposition lemma for metrics into a finite sum of primitive metrics (in the form of Lemma 3 in \cite{CDS12}; the original version is Lemma 1 in \cite{Nash54}):

\begin{lemma}\label{l:decompose}
Let $G_0\in \R^{n\times n}$ be a symmetric positive definite matrix. There exists $r>0$, vectors $\xi_1,\dots,\xi_{n_*}\in \mathbb S^{n-1}$ and linear maps $L_i:\R^{n\times n}_{sym}\to\R$ such that $G=\sum_{i=1}^{n_*}L_i(G)\xi_i\otimes \xi_i$
and moreover $L_i(G)\geq r$ for every $i$ and for every $G\in \R^{n\times n}_{sym}$ with $|G-G_0|\leq r$. Here $n_*=\frac{1}{2}n(n+1)$.
\end{lemma}

As in \cite{CDS12} we can argue by compactness to claim that, for any $\gamma>0$ there exists $r_0>0$ such that Lemma \ref{l:decompose} holds with $r=4r_0$ for any $G_0$ satisfying $\frac{1}{\gamma}\textrm{Id}\leq G_0\leq \gamma\textrm{Id}$.

\bigskip

For the case $n=2$ we have $n_*=3$. However, as in \cite{DIS15} we can use conformal coordinates to reduce the sum in Lemma \ref{l:decompose} to only 2 terms. For our purpose we need the existence of conformal coordinates in the following quantitative formn (c.f.~Proposition 2.3 in \cite{DIS15}):

\begin{proposition}\label{p:conformal}
Let $\Omega\subset\R^2$ be a simply connected open bounded set with smooth boundary and $H:\Omega\to \R^{2\times 2}$ a smooth $2\times2$ symmetric positive definite matrix valued function such that, for some $0<\alpha<1$ and $\gamma,M\geq 1$
\begin{equation}\label{e:p-condition}
\frac{1}{\gamma}\textrm{Id}\leq H\leq \gamma\textrm{Id},\quad \|H\|_{C^\alpha(\Omega)}\leq M.
\end{equation}
Then there exists a smooth diffeomorphism $\Phi:\Omega\to\R^2$ and a smooth positive function $\vartheta:\Omega\to\R$ satisfying
\begin{equation}\label{e:isothermal}
H=\vartheta^2(\nabla\Phi_1\otimes\nabla\Phi_1+\nabla\Phi_2\otimes\nabla\Phi_2).
\end{equation}
Moreover the following estimates hold:
\begin{equation*}
\begin{split}
&\det(D\Phi(x))\geq c_0,\quad \vartheta(x)\geq c_0\quad \textrm{ for all }x\in\Omega,\\
&\|\vartheta\|_{j+\alpha}+\|\nabla\Phi\|_{j+\alpha}\leq C_j\|H\|_{j+\alpha}, j\in\mathbb{N},
\end{split}
\end{equation*}
where the constants $c_0>0$, $C_j\geq 1$ depend only on $\alpha, \gamma, M$ and on $\Omega$.
\end{proposition}

\begin{proof}
A classical computation (see e.g.~\cite{Chern:1955jf}) shows that \eqref{e:isothermal} is equivalent to the linear Beltrami equation
\begin{equation}\label{e:beltrami}
\partial_{\bar{z}}\Phi=\mu\partial_z\Phi, \text{ with } \mu=\frac{H_{11}-H_{22}+2iH_{12}}{H_{11}+H_{22}+2\sqrt{\det H}}.
\end{equation}
using complex notation with $z=x_1+ix_2$, together with
\begin{equation}\label{e:conformal}
\vartheta^2=\frac{\sqrt{\det H}}{\det D\Phi}.
\end{equation}
Indeed, writing \eqref{e:isothermal} in differential form notation as
$$
H_{11}dx_1^2+2H_{12}dx_1dx_2+H_{22}dx_2^2=\vartheta^2|d\Phi|^2,
$$
we express the left hand side with complex differentials as
$\lambda |dz+\mu d\bar{z}|^2$, and obtain the expression for $\mu$ in \eqref{e:beltrami}. Then \eqref{e:conformal} follows by taking the determinant of both sides of \eqref{e:isothermal}.

Observe that
$$
|\mu|^2=\frac{(\tr H)^2-4\det H}{(\tr H)^2+4\det H+4\sqrt{\det H}\tr H}\leq 1-\frac{4\det H}{(\tr H)^2},
$$
so that, using \eqref{e:p-condition} we deduce $|\mu|^2\leq 1-\tfrac{1}{\gamma^4}$. Extend $\mu$ to $\R^2\simeq\C$ to obtain $\mu\in C^{\infty}_c(\C)$ with
$$
|\mu|^2\leq 1-\tfrac{1}{2\gamma^4},\quad \|\mu\|_{C^\alpha(\C)}\leq C(M).
$$

Define $\Phi$ to be the principal solution of \eqref{e:beltrami} (i.e. normalized by $\Phi(0)=0$ and $\Phi(1)=1$). The lower bound on the Jacobian and the Schauder estimates follow as in \cite{ACFJK}.
\end{proof}

\subsection{Short immersions}

The essence of the Nash-Kuiper construction is to start with a strictly short immersion, and successively reduce the metric error via a sequence of short immersions. Given an immersion $u:\mathcal{M}\to\R^m$ of class $C^1$ we denote by $u^\sharp e$ the pullback of the Euclidean metric $e$ from $\R^m$ onto $\mathcal{M}$. Thus, if $g$ is a metric on $\mathcal{M}$, the map $u$ is isometric if and only if $u^\sharp e=g$. In local coordinates $(u^\sharp e)_{ij}=\partial_i u\cdot\partial_j u$.    We recall the definition of short immersion:

\begin{definition}(\textit{Short immersion})
Following \cite{Nash54,Gr86}, a $C^1$ immersion $u:\mathcal{M}\rightarrow\R^m$ is called \emph{short} if
$$
g-u^\sharp e\geq 0\quad\textrm{ on }\mathcal{M},
$$
and \emph{strictly short} if
$$
g-u^\sharp e>0\quad\textrm{ on }\mathcal{M}.
$$
\end{definition}

\begin{definition}\label{d:strong} (\textit{Strong short immersion})
We call a $C^1$ immersion $u:\mathcal{M}\rightarrow\R^m$ \emph{strongly short} if
$$
g-u^\sharp e=\rho^2(g+h)
$$
with a non-negative function $\rho\in C(\mathcal{M})$ and symmetric tensor $h\in C(\mathcal{M};\R^{n\times n}_{sym})$
satisfying
$$
-\tfrac{1}{2}g\leq h\leq \tfrac{1}{2}g\quad\textrm{ on }\mathcal{M}.
$$
\end{definition}


Finally, we define adapted short immersions analogously to \cite{CS19} (see also \cite{DS17}).
\begin{definition}\label{d:adapt} (\textit{Adapted short immersion})
Given a closed subset $\Sigma\subset\mathcal{M}$ a strongly short immersion $u:\mathcal{M}\rightarrow\R^m$ is called \emph{adapted short with respect to $\Sigma$ with exponent $0<\theta<1$} if $u\in C^{1, \theta}(\mathcal{M})$ is strongly short with
$$
g-u^\sharp e=\rho^2(g+h)
$$
such that $\Sigma=\{\rho=0\}$,
$$
u\in C^2(\mathcal{M}\setminus \Sigma),\quad \rho,h\in C^1(\mathcal{M}\setminus\Sigma)
$$
and there exists a constant $A\geq 1$ such that, in any chart $\Omega_k$
\begin{equation}
\label{e:adapt-final-2}
\begin{split}
&|\nabla^2u(x)|\leq A\rho(x)^{1-\tfrac{1}{\theta}}, \\
&|\nabla\rho(x)|\leq A\rho(x)^{1-\tfrac{1}{\theta}},\\
&|\nabla h(x)|\leq A\rho(x)^{-\tfrac{1}{\theta}}.
\end{split}
\end{equation}
for any $x\in\Omega_k\setminus \Sigma$.
\end{definition}

The motivation for the estimates in \eqref{e:adapt-final-2} is as follows: From the pointwise estimate on $\rho$ we deduce that $\rho^{\frac{1}{\theta}}$ is Lipschitz continuous on $\mathcal{M}$, hence $\rho$ is H\"older continuous with exponent $\theta$. Moreover, since the metric error is of order $\rho$, one may hope to remove this error and obtain an isometric immersion $v$ with $|\nabla u-\nabla v|\sim {\rho}$ whilst keeping $|\nabla^2u-\nabla^2 v|\lesssim \rho^{1-\tfrac{1}{\theta}}$. Consequently we would obtain
$$
\|u-v\|_{1,\theta}\lesssim \|u-v\|_1^{1-\theta}\|u-v\|_2^{\theta}\lesssim 1.
$$
In other words, the estimates \eqref{e:adapt-final-2} are consistent with the goal of removing a metric error of amplitude $\theta$ whilst keeping $\nabla u$ H\"older continuous with exponent $\theta$.

\section{Iteration propositions}\label{s:iterate}

In this section we consider the local perturbation problem, called a ``step" in \cite{Nash54}. Throughout the section we work in an open bounded subset $\Omega\subset\R^n$ and consider immersions $u:\Omega\to\R^{n+1}$. In coordinates the pullback metric via an immersion $u$ is given by $(\partial_iu\cdot\partial_ju)_{i,j=1\dots n}$, for which we shall use the matrix notation $\nabla u^T\nabla u$. Recall that $(\nabla u)_{ki}=\partial_{i}u^k$ is the Jacobian matrix of the map $u$. The propositions below are minor modifications of analogous results in  \cite{CDS12,DIS15,CS19}. The key difference is that we consider compactly supported perturbations of the metric with general primitive metrics of the form $\rho^2(d\Phi)^2$ (equivalently, in matrix notation $\rho^2\nabla\Phi\otimes\nabla\Phi$). To keep the value of initial immersion out of the support of $\rho$ and at the same time handle loss of derivatives, we mollify the metric and the immersion at different length scales.

\subsection{Adding primitive metrics - a single step}

 \begin{proposition}\label{p:step} {\bf[Step]}
Let $u\in C^{2}(\Omega, \R^{n+1})$ be an immersion, and $\varrho, \Phi\in C^{2}(\Omega)$. Assume that
\begin{align}
\frac{1}{\gamma}\textrm{Id}\leq\nabla u^T\nabla u\leq\gamma \textrm{Id},&\quad \frac{1}{M}\leq |\nabla\Phi|\leq M\quad\textrm{ in }\Omega,\label{e:step-1a}\\
\|u\|_{2}&\leq M\delta^{1/2}\nu,\label{e:step-2a}\\
\|\varrho\|_0\leq M\eps^{1/2},\quad \|\varrho\|_1&\leq M\eps^{1/2}\nu, \quad \|\varrho\|_2\leq M\eps^{1/2}\nu\tilde\nu,\label{e:step-3a}\\
\|\nabla\Phi\|_1&\leq M\nu, \quad \|\nabla\Phi\|_2\leq M\nu\tilde\nu \label{e:step-4a}
\end{align}
for some $M, \gamma\geq 1$, $\eps\leq\delta\leq 1$ and $\nu\leq\tilde\nu$.
There exists a constant $c_0=c_0(M, \gamma)$ such that, for any
\begin{equation}\label{e:constraint-on-lambda}
\lambda\geq c_0\frac{\delta^{1/2}}{\eps^{1/2}}\tilde\nu,
\end{equation}
there exists an immersion $v\in C^{2}(\Omega, \R^{n+1})$ such that
\begin{align}
\frac{1}{\bar\gamma}\textrm{Id}\leq\nabla v^T\nabla v&\leq \bar\gamma \textrm{Id}\quad\textrm{ in }\Omega, \label{e:step-0}\\
v&=u\textrm{ on }\Omega\setminus \textrm{supp }\varrho,\label{e:step-1}\\
\|v-u\|_j&\leq C\eps^{1/2}\lambda^{j-1}, j=0, 1, \label{e:step-2}\\
\|v\|_2&\leq C\eps^{1/2}\lambda,\label{e:step-3}\\
\|\nabla v^T\nabla v-(\nabla u^T\nabla u+\varrho^2\nabla\Phi\otimes\nabla\Phi)\|_j&\leq C\eps^{1/2}\delta^{1/2}\nu\lambda^{j-1}, j=0, 1.\label{e:step-4}
\end{align}
Here $\bar\gamma,C\geq 1$ are constants depending only on $M,\gamma$.
\end{proposition}


\begin{proof}
We proceed as in the proof of Proposition 3.1 in \cite{CS19}. In the following we will denote by $C$ a generic constant whose value may change from line to line, but depends only on $M,\gamma$. 

Fix $\lambda\geq c_0\frac{\delta^{1/2}}{\eps^{1/2}}\tilde\nu\geq\tilde\nu,$
with $c_0$  still to be chosen. Initially we only assume that $c_0\geq 1$ so that in particular we may assume the inequalities
$$
\nu\leq \frac{\delta^{1/2}}{\eps^{1/2}}\nu\leq \lambda.
$$
We regularize $u$ on length-scale $\lambda^{-1}$ to obtain a smooth immersion $\tilde{u}$ satisfying
\begin{equation}\label{e:tildeu-C2}
\|\tilde{u}-u\|_1\leq C(M)\delta^{1/2}\nu\lambda^{-1} ,\quad \|\tilde{u}\|_{2}\leq C(M)\delta^{1/2}\nu,\quad \|\tilde{u}\|_{3}\leq C(M)\delta^{1/2}\nu\lambda.
\end{equation}
Since
$$
\nabla\tilde u^T\nabla\tilde u=\nabla u^T\nabla u-(\nabla u-\nabla\tilde u)^T\nabla u-\nabla\tilde u^T(\nabla u-\nabla\tilde u),
$$
we have
\begin{equation}\label{e:tildeu-C1}
\frac{1}{2\gamma}\textrm{Id}\leq \nabla\tilde u^T\nabla \tilde u\leq 2\gamma \textrm{Id},
\end{equation}
provided $\delta^{1/2}\nu\lambda^{-1}\leq c_2^{-1}$ for some $c_2=c_2(M,\gamma)$. Choosing $c_0\geq c_2$ in the inequality \eqref{e:constraint-on-lambda} constraining $\lambda$ will ensure this. Then it follows that $\nabla\tilde u^T\nabla \tilde u$ is invertible, and hence we can set
\begin{align*}
&\tilde\xi=\nabla \tilde{u}(\nabla \tilde{u}^T\nabla \tilde{u})^{-1}\nabla\Phi,\quad \xi=\frac{\tilde\xi}{|\tilde\xi|^2},\\
&\tilde\zeta=*(\partial_{x_1}\tilde{u}\wedge\partial_{x_2}\tilde{u}\wedge\cdots\wedge\partial_{x_n}\tilde{u}),\quad
\zeta=\frac{\tilde\zeta}{|\tilde\zeta||\tilde\xi|},\\
&\tilde{\varrho}=|\tilde\xi|\varrho.
\end{align*}
Directly from construction we have
\begin{equation}\label{e:xizeta}
\nabla \tilde u^T\xi=\frac{1}{|\tilde\xi|^2}\nabla\Phi,\quad \nabla\tilde u^T\zeta=\nabla\tilde u^T\tilde\zeta=0.
\end{equation}
It follows from \eqref{e:tildeu-C2}-\eqref{e:tildeu-C1} and \eqref{e:step-3a} that
\begin{equation}\label{e:xizeta-est}
\begin{split}
\|(\xi, \zeta)\|_0&\leq C, \\
\|(\xi,\zeta)\|_1&\leq C\nu,\\
\|(\xi,\zeta)\|_2&\leq C\nu(\delta^{1/2}\lambda+\tilde\nu)\leq C\nu\lambda,\\
\end{split}
\end{equation}
and
\begin{equation}\label{e:acm}
\begin{split}
\|\tilde{\varrho}\|_0&\leq C\eps^{1/2},\\
\|\tilde{\varrho}\|_1&\leq C(\|\varrho\|_1\|~|\tilde\xi|~\|_0+\|\varrho\|_0\|~|\tilde\xi|~\|_1)\\
&\leq C\eps^{1/2}\nu, \\
\|\tilde{\varrho}\|_2&\leq C(\|\varrho\|_2\|~|\tilde\xi|~\|_0+\|\varrho\|_0\|~|\tilde\xi|~\|_2)\\
&\leq C(\eps^{1/2}\nu\tilde\nu+\eps^{1/2}\nu\lambda)\\
&\leq C\eps^{1/2}\nu\lambda,
\end{split}
\end{equation}
where we have used that
$\tilde\nu\leq \lambda$.
We define
$$
v=u+\frac{1}{\lambda}\bigl(\Gamma_1(\tilde{\varrho}, \lambda \Phi)\xi+\Gamma_2(\tilde{\varrho}, \lambda\Phi)\zeta\bigr).
$$
Using Lemma \ref{l:gamma} we see that $v=u$ outside $\textrm{supp }\varrho$, so that \eqref{e:step-1} holds.

As in \cite{CS19} we have
\begin{equation}\label{e:interu}
\begin{split}
\|v-u\|_{j}&\leq\frac{1}{\lambda}(\|\Gamma_1\xi\|_{j}+\|\Gamma_2\zeta\|_{j})\\
&\leq\frac{C}{\lambda}(\|\Gamma_1\|_{j}\|\xi\|_{0}+\|\Gamma_1\|_0\|\xi\|_{j}+
\|\Gamma_2\|_{j}\|\zeta\|_0+\|\Gamma_2\|_0\|\zeta\|_{j})
\end{split}
\end{equation}
for $j=0, 1, 2$. Therefore we need to estimate $\|\Gamma_i\|_{j}$ for  $i=1,2$ and $j=0, 1, 2$, where we refer to the $C^j$-norms in $x\in\Omega$ of the composition $x\mapsto \Gamma_i(\tilde{\varrho}(x), \lambda \Phi(x))$.
Using Lemma \ref{l:gamma} and the assumptions \eqref{e:step-1a}-\eqref{e:step-4a} we deduce that
\begin{equation}\label{e:Gammaest1}
\begin{split}
\|\Gamma_1\|_0+\|\partial_t\Gamma_1\|_0&+\|\partial_t^2\Gamma_1\|_0\leq C\|\tilde{\varrho}^2\|_0\leq C\eps,\\
\|\Gamma_1\|_1&\leq \|\partial_t\Gamma_1\|_0\|\nabla\Phi\|_0\lambda+\|\partial_s\Gamma_1\|_0\|\nabla\tilde{\varrho}\|_0\\
&\leq C\eps\lambda+C(M,\gamma)\eps\nu\\
&\leq C\eps\lambda,\\
\|\partial_t\Gamma_1\|_1&\leq \|\partial_t^2\Gamma_1\|_0\|\nabla\Phi\|_0\lambda+\|\partial_s\partial_t\Gamma_1\|_0\|\nabla\tilde{\varrho}\|_0\\
&\leq C\eps\lambda,
\end{split}
\end{equation}
and
\begin{equation}\label{e:Gammaest2}
\begin{split}
\|\Gamma_2\|_0+\|\partial_t\Gamma_2\|_0&+\|\partial_t^2\Gamma_2\|_0\leq C\|\tilde{\varrho}\|_0\leq C\eps^{1/2},\\
\|\Gamma_2\|_1&\leq \|\partial_t\Gamma_2\|_0\|\nabla\Phi\|_0\lambda+\|\partial_s\Gamma_2\|_0\|\nabla\tilde{\varrho}\|_0\\
&\leq C\eps^{1/2}\lambda+C(M, \gamma)\eps^{1/2}\nu\\
&\leq C\eps^{1/2}\lambda,\\
\|\partial_t\Gamma_2\|_1&\leq \|\partial_t^2\Gamma_2\|_0\|\nabla\Phi\|_0\lambda+\|\partial_s\partial_t\Gamma_2\|_0\|\nabla\tilde{\varrho}\|_0\\
&\leq C\eps^{1/2}\lambda,
\end{split}
\end{equation}
where we have used that $\lambda\geq C(M, \gamma)\nu$ - this can be guaranteed by an appropriate choice of $c_0=c_0(M,\gamma)$.
Similarly, we also have
\begin{equation}\label{e:Gammaest12}
\begin{split}
\|\partial_s\Gamma_1\|_0&\leq C\|\tilde{\varrho}\|_0\leq C\eps^{1/2},\\
\|\partial_s\Gamma_2\|_0&\leq C,\\
\|\partial_s\Gamma_1\|_1&\leq \|\partial_t\partial_s\Gamma_1\|_0\|\nabla\Phi\|_0\lambda+\|\partial_s^2\Gamma_1\|_0\|\nabla\tilde{\varrho}\|_0\leq C\eps^{1/2}\lambda,\\
\|\partial_s\Gamma_2\|_1&\leq \|\partial_t\partial_s\Gamma_2\|_0\|\nabla\Phi\|_0\lambda+\|\partial_s^2\Gamma_2\|_0\|\nabla\tilde{\varrho}\|_0\leq C\lambda.
\end{split}
\end{equation}
Thus by \eqref{e:Gammaest1}-\eqref{e:Gammaest12}, we derive
\begin{align*}
\|v-u\|_0&\leq C\eps^{1/2}\lambda^{-1},\\
\|v-u\|_1&\leq C\eps^{1/2}+C\eps^{1/2}\nu\lambda^{-1}\\
&\leq C\eps^{1/2},\\
\|v-u\|_2&\leq C\eps^{1/2}\lambda+C\eps^{1/2}\nu\\
&\leq C\eps^{1/2}\lambda.
\end{align*}
Summarizing, we arrive at \eqref{e:step-2}, and since $\eps^{1/2}\lambda\geq M\delta^{1/2}\nu$, also at \eqref{e:step-3}.

\smallskip

We next derive estimates on the metric error, as in \cite{CS19}. We have
\begin{align*}
\nabla v=&\nabla u+(\partial_t\Gamma_1\xi\otimes\nabla\Phi+\partial_t\Gamma_2\zeta\otimes\nabla\Phi)+
\frac{1}{\lambda}(\Gamma_1\nabla \xi+\Gamma_2\nabla \zeta)\\
&+\frac{1}{\lambda}(\partial_s\Gamma_1\xi\otimes\nabla\tilde{\varrho}+\partial_s\Gamma_2\zeta\otimes\nabla\tilde{\varrho})\\
=&\nabla u+B+E_1+E_2,
\end{align*}
where we have set
$$
B=\partial_t\Gamma_1 \xi\otimes\nabla\Phi+\partial_t\Gamma_2\zeta\otimes\nabla\Phi\,,
$$
$E_1=E_1^{(1)}+E_1^{(2)}$ with
$$
E_1^{(1)}=\frac{1}{\lambda}\left(\Gamma_1\nabla \xi+\frac{\Gamma_2\nabla\tilde\zeta}{|\tilde\zeta||\tilde\xi|}-
\frac{\Gamma_2\tilde\zeta\otimes\nabla|\tilde\zeta|}{|\tilde\zeta|^2|\tilde\xi|}\right),\quad
E_1^{(2)}=-\frac{\Gamma_2}{\lambda}\frac{\tilde\zeta\otimes\nabla|\tilde\xi|}{|\tilde\zeta||\tilde\xi|^2},
$$
where we have calculated that
$$
\nabla \zeta=\frac{\nabla\tilde\zeta}{|\tilde\zeta||\tilde\xi|}-\frac{\tilde\zeta\otimes\nabla|\tilde\zeta|}{|\tilde\zeta|^2|\tilde\xi|}-\frac{\tilde\zeta\otimes\nabla|\tilde\xi|}{|\tilde\zeta||\tilde\xi|^2}.
$$
and $E_2=E_2^{(1)}+E_2^{(2)}$ with
$$
E_2^{(1)}=\frac{1}{\lambda}\partial_s\Gamma_1\xi\otimes\nabla\tilde{\varrho},\quad
E_2^{(2)}=\frac{1}{\lambda}\partial_s\Gamma_2\zeta\otimes\nabla\tilde{\varrho}.
$$
Using \eqref{e:xizeta} and Lemma \ref{l:gamma}, we have
$$	
\nabla \tilde{u}^TB+B^T\nabla\tilde{u}+B^TB=\varrho^2\nabla\Phi\otimes\nabla\Phi
$$
and
$$
\nabla\tilde u^T(E_1^{(2)}+E_2^{(2)})=0.
$$
Therefore we may write, using the notation $\textrm{sym}(H)=(H+H^T)/2$,
\begin{equation}\label{e:metricerror1}
\begin{split}
&\nabla  v^T\nabla v-(\nabla u^T\nabla u+\varrho^2\nabla\Phi\otimes\nabla\Phi)=2\textrm{sym}\left[(\nabla u-\nabla\tilde{u})^T(B+E_1^{(2)}+E_2^{(2)})\right]\\
&\quad +2\textrm{sym}\left[B^T(E_1^{(2)}+E_2^{(2)})\right]+2\textrm{sym}\left[(\nabla u+B)^T(E_1^{(1)}+E_2^{(1)})\right]\\
&\quad +(E_1+E_2)^T(E_1+E_2).
\end{split}
\end{equation}
Using the estimates \eqref{e:xizeta-est}, \eqref{e:acm}, \eqref{e:Gammaest1}, \eqref{e:Gammaest2} and \eqref{e:Gammaest12} we obtain
\begin{align*}
\|B\|_0&\leq C\eps^{1/2},\\
\|E_1^{(1)}\|_0&\leq C\lambda^{-1}(\eps\nu+\eps^{1/2}\delta^{1/2}\nu)\leq C\lambda^{-1}\eps^{1/2}\delta^{1/2}\nu,\\	
\|E_2^{(1)}\|_0&\leq C\lambda^{-1}\eps\nu,\\
\|E_1^{(2)}\|_0+\|E_2^{(2)}\|_0&\leq C\lambda^{-1}\eps^{1/2}\nu.
\end{align*}
Recall that $\eps\leq\delta$.
Using also that $\lambda\geq \nu$, we deduce
\begin{equation}\label{e:metricerror0}
\|\nabla v^T\nabla v-(\nabla u^T\nabla u+\varrho^2\nabla\Phi\otimes\nabla\Phi)\|_0\leq C\eps^{1/2}\delta^{1/2}\lambda^{-1}\nu.
\end{equation}
Similarly, using the Leibniz-rule as in \cite{CS19}, we obtain
\begin{align*}
\|B\|_1&\leq C\eps^{1/2}\lambda,\\
\|E_{1}^{(1)}\|_1&\leq C\eps^{1/2}\delta^{1/2}\nu,\\
\|E_{2}^{(1)}\|_1&\leq C\eps\nu,\\
\|E_1^{(2)}\|_1+\|E_{2}^{(2)}\|_1&\leq C\eps^{1/2}\nu,
\end{align*}
Therefore, as in \cite{CS19}, after differentiating \eqref{e:metricerror1} we deduce
\begin{equation*}
\|\nabla v^T\nabla v-(\nabla u^T\nabla u+\varrho^2\nabla\Phi\otimes\nabla\Phi)\|_{1}\leq C\eps^{1/2}\delta^{1/2}\nu.
\end{equation*}
This verifies \eqref{e:step-4}.

Finally, we verify that $v$ is a bounded  immersion. From \eqref{e:metricerror0} it follows that
\begin{align*}
\|\nabla v^T\nabla v-(\nabla u^T\nabla u+\varrho^2\nabla\Phi\otimes\nabla\Phi)\|_0\leq \frac{1}{2\gamma},
\end{align*}
provided we choose $c_0\geq 2\gamma C(M, \gamma)$. Using \eqref{e:step-3a} and $\eps\leq 1$ we observe
$$
0\leq \varrho^2\nabla\Phi\otimes\nabla\Phi\leq M^2\textrm{Id},
$$ so that
from \eqref{e:step-1a} we readily deduce \eqref{e:step-0}.
This concludes the proof.
\end{proof}


\subsection{Adding general metrics - a stage}

Next, as in Proposition 3.2 from \cite{CS19}, we can apply Proposition \ref{p:step} $N$ times successively to add a term of the form
$$
\sum_{k=1}^N\varrho_k^2\nabla\Phi_k\otimes\nabla\Phi_k
$$
to the metric. Since the proof is exactly the same, here we only recall the statement without proof.
\begin{proposition}\label{p:stage}{\bf [Stage]}
Let $u\in C^2(\Omega, \R^{n+1})$ be an immersion, $\varrho_k, \Phi_k\in C^2(\Omega)$ for $k=1,\cdots, N. $  Assume that
\begin{align*}
\frac{1}{\gamma}\textrm{Id}&\leq\nabla u^T\nabla u\leq\gamma\textrm{Id},\quad \frac{1}{M}\leq|\nabla\Phi_k|\leq M,\, \text{ in } \Omega,\\
&\|u\|_2\leq M\delta^{1/2}\nu,\\
\|\varrho_k\|_0\leq M\eps^{1/2}, \quad &\|\varrho_k\|_1\leq M\eps^{1/2}\nu, \quad \|\varrho_k\|_2\leq M\eps^{1/2}\nu\tilde{\nu},\\
 \|\nabla\Phi_k\|_1&\leq M\nu, \quad \|\nabla\Phi_k\|_2\leq M\nu\tilde\nu,
\end{align*}
for some $M, \gamma>1,$ $\eps\leq\delta\leq1$ and $\nu\leq\tilde{\nu}.$ There exists a constant $c_1=c_1(M, \gamma)$ such that for any $K>c_1\tilde\nu\nu^{-1}$ there exists an immersion $v\in C^{1}(\Omega, \R^{n+1})$ such that
\begin{align*}
v&=u \text{ on } \Omega\setminus\bigcup_{k=1}^N\textrm{supp}\varrho_k,\\
\|v-u\|_j&\leq C\eps^{1/2}(\eps^{-1/2}\delta^{1/2}\nu K)^{j-1}, j=0, 1,\\
\|v\|_2&\leq C\delta^{1/2}\nu K^N\,.
\end{align*}
Furthermore, there exists $\mathcal{E}\in C^1(\Omega, \R^{n\times n}_{sym})$ such that
\begin{align*}
\nabla v^T\nabla v=\nabla u^T\nabla u+\sum_{k=1}^N\varrho_k^2\nabla\Phi_k\otimes\nabla\Phi_k+\mathcal{E} \text{ in }\Omega
\end{align*}
with
\begin{align*}
\|\mathcal{E}\|_0&\leq C\eps K^{-1},\\
\|\mathcal{E}\|_1&\leq C\eps^{1/2}\delta^{1/2}\nu K^{N-1}.
\end{align*}
Here $C\geq 1$ is a constant depending only on $N,M,\gamma$.
\end{proposition}


\bigskip

In the 2-dimensional case we can combine Proposition \ref{p:stage} with Proposition \ref{p:conformal} to add a term  of the form  $\rho^2(G+H)$ to the metric. Let us recall that $G$ is assumed to be a given $C^1$ metric, expressed in local coordinates of some chart, identified with an open bounded subset $\Omega\subset\R^2$.

\begin{corollary}\label{c:metric-1}
Let $G$ be a $C^1$ metric on $\Omega\subset\R^2$ with $\tfrac{1}{\gamma}\textrm{Id}\leq G\leq \gamma\textrm{Id}$ and $\|G\|_1\leq \gamma$ for some $\gamma\geq 1$. Let $u\in C^2(\Omega, \R^{3})$ an immersion, $\rho\in C^1(\Omega)$, $H\in C^1(\Omega;\R^{2\times 2}_{sym})$ such that
\begin{align*}
\frac{1}{\gamma}\textrm{Id}\leq \nabla u^T&\nabla u\leq\gamma \textrm{Id}\,\textrm{ in }\Omega,\\
\|u\|_{2}&\leq \delta^{1/2}\lambda,
\end{align*}
and
\begin{align}
&\|\rho\|_0\leq \delta^{1/2}, \quad \|\rho\|_1\leq \delta^{1/2}\lambda, \label{e:rho-assumption}\\
&\|H\|_0\leq \lambda^{-\alpha}, \quad \|H\|_1\leq \lambda^{1-\alpha}\,,\label{e:H-assumption}
\end{align}
for some $0<\delta<1$, $\alpha>0$ and $\lambda>1$ such that $2\gamma\leq \lambda^\alpha$. Then, for any $\kappa\geq 1$
there exists an immersion $v\in C^2(\Omega; \R^{3})$ and $\mathcal{E}\in C^1(\Omega;\R^{2\times 2}_{sym})$ such that
\begin{equation}\label{e:error-support}
\begin{split}
 \nabla v^T\nabla v=&\nabla u^T\nabla u+\rho^2(G+H)+\mathcal{E}\quad\textrm{ in }\Omega,\\
v=&u\textrm{ on }\supp\rho+B_{\lambda^{-\kappa}}(0)
\end{split}
\end{equation}
with estimates
\begin{align}
\|v-u\|_0&\leq C\delta^{1/2}\lambda^{-\kappa},\label{e:stage-2}\\
\|v-u\|_1&\leq C\delta^{1/2},\label{e:stage-3}\\
\|v\|_2&\leq C\delta^{1/2}\lambda^{2\kappa-1},\label{e:stage-4}
\end{align}
and
\begin{equation}\label{e:stage-5}
\|\mathcal{E}\|_0\leq C\delta\lambda^{1-\kappa}, \quad
\|\mathcal{E}\|_1\leq C\delta\lambda^{\kappa}.
\end{equation}
Here $C\geq 1$ is a constant depending only on $\gamma$ and $\alpha$.
\end{corollary}

\begin{proof}
We start by mollifying $\rho$, $G$ and $H$ at length-scale $\ell$, with
\begin{equation}\label{e:ell}
\ell=\lambda^{-\kappa}.
\end{equation}
In this way we obtain $\tilde \rho$, $\tilde G$ and $\tilde H$ such that, using \eqref{e:rho-assumption}, \eqref{e:H-assumption} and Proposition \ref{p:mollification},
\begin{align*}
\|\tilde \rho\|_0\leq \delta^{1/2}, \quad & \|\tilde H\|_{0}\leq \lambda^{-\alpha},\\
\|\tilde \rho\|_j\leq C_j\delta^{1/2}\lambda\ell^{1-j},\quad & \|\tilde H\|_{j}\leq C_j\lambda^{1-\alpha}\ell^{1-j},\\
\|\tilde \rho-\rho\|_0\leq C\delta^{1/2}\lambda\ell,\quad  &\|\tilde H-H\|_0\leq C\lambda^{1-\alpha}\ell\,\\
\|\tilde G-G\|_0\leq C(\gamma)\ell,\quad &\|\tilde G\|_{j}\leq C_j(\gamma)\ell^{1-j}
\end{align*}
for all $j\geq 1$.
In particular using  \eqref{e:inerpolation}  we deduce, for any $0<\alpha'<1$
\begin{equation*}
\begin{split}	
\|\tilde G+\tilde H\|_{\alpha'}&\leq \|\tilde G\|_{\alpha'}+\|\tilde H\|_{\alpha'}\\
&\leq C(\gamma+\lambda^{\alpha'-\alpha})\\
&\leq C\gamma
\end{split}
\end{equation*}
provided $\alpha'\leq \alpha$.
Since positive definite matrices form a convex set, we have $\tfrac{1}{\gamma}\textrm{Id}\leq \tilde G\leq \gamma\textrm{Id}$. Therefore
$$
\tilde G+\tilde H\geq (\tfrac{1}{\gamma}-\lambda^{-\alpha})\textrm{Id}\geq \tfrac{1}{2\gamma}\textrm{Id},
$$
since we assumed $\lambda^{\alpha}\geq 2\gamma$ 
Let $\mathbf{h}=\rho^2(G+H)$ and $\tilde{\mathbf{h}}=\tilde \rho^2(\tilde G+\tilde H)$. 
We apply Proposition \ref{p:conformal} to obtain
$$
\tilde G+\tilde H=\tilde\vartheta^2(\nabla\Phi_1\otimes\nabla\Phi_1+\nabla\Phi_2\otimes\nabla\Phi_2),
$$
with $\tilde\vartheta, \Phi=(\Phi_1, \Phi_2)$ satisfying
\begin{equation*}
\tilde{\vartheta}\geq C(\gamma), \,\det(D\Phi)\geq C(\gamma),
\end{equation*}
and, using once again \eqref{e:inerpolation},
\begin{align*}
\|\tilde{\vartheta}\|_{j+\alpha'}+\|\nabla\Phi\|_{j+\alpha'}\leq C_j(\alpha',\gamma)\lambda^{1-\alpha}\ell^{1-j-\alpha'}.
\end{align*}
Choosing $\alpha'\leq \frac{\alpha}{\kappa}$ and recalling \eqref{e:ell} we then deduce 
\begin{equation}\label{e:phi-estimate}
\begin{split}
\frac{1}{C(\gamma)}\leq|\nabla\Phi_1|&, \,|\nabla\Phi_2|\leq C(\gamma);\\
\|\nabla\Phi\|_{j}&\leq C_j(\gamma)\lambda\ell^{1-j}\,,\\
\|\tilde\vartheta\|_{j}&\leq C_j(\gamma)\lambda\ell^{1-j}\,
\end{split}
\end{equation}
for any $j\geq 1$.
Define
$\vartheta=\tilde \vartheta \tilde \rho$
so that
\begin{equation}\label{e:decompstage}
\tilde{\mathbf{h}}=\vartheta^2(\nabla\Phi_1\otimes\nabla\Phi_1+\nabla\Phi_2\otimes\nabla\Phi_2)
\end{equation}
and $\vartheta$ satisfies
\begin{equation}\label{e:vartheta}
\begin{split}
\|\vartheta\|_0&\leq C(\gamma)\delta^{1/2},\\
\|\vartheta\|_1&\leq C(\gamma)\delta^{1/2}\lambda,\\
\|\vartheta\|_2&\leq  C(\gamma)\delta^{1/2}\lambda\ell^{-1}\,.
\end{split}
\end{equation}
We remark that
$$
\textrm{dist}\left(\textrm{supp }\tilde{\mathbf{h}},\textrm{supp }\rho\right)\leq \ell=\lambda^{-\kappa}.
$$
Then from \eqref{e:phi-estimate}-\eqref{e:vartheta}, in order to add $\tilde{\mathbf{h}}$, we  apply Proposition \ref{p:stage} with
$$
\eps=\delta, \quad \nu=\lambda, \quad \tilde\nu=\ell^{-1}.
$$
In particular we fix
$$
K=c_1(\gamma)\tilde\nu\nu^{-1}=c_1(\gamma)\lambda^{\kappa-1}
$$
with $c_1$ from Proposition \ref{p:stage}. We then conclude the existence of an
immersion $v\in C^2(\Omega;\R^{3})$ such that
\begin{align*}
v=u&\textrm{ on }\Omega\setminus \textrm{supp }\tilde{\mathbf{h}}\\
\|v-u\|_0&\leq C(\gamma)\delta^{1/2}\lambda^{-\kappa}\,,\\
\|v-u\|_1&\leq C\delta^{1/2},\\
\|v\|_2&\leq C\delta^{1/2}\lambda^{2\kappa-1}.
\end{align*}
This shows \eqref{e:error-support}-\eqref{e:stage-4}. Moreover, the new metric error
\begin{align*}
\mathcal{E}&=\nabla v^T\nabla v-(\nabla u^T\nabla u+\rho(G+H))
\end{align*}
satisfies
\begin{align*}
\|\mathcal{E}\|_0&\leq C\delta \lambda^{1-\kappa}+\|\mathbf{h}-\tilde{\mathbf{h}}\|_0\\
\|\mathcal{E}\|_1&\leq C\delta\lambda^{\kappa}+\|\mathbf{h}-\tilde{\mathbf{h}}\|_1\,.
\end{align*}
Finally, using \eqref{e:ell}, we estimate
\begin{equation*}
\begin{split}	
\|\tilde {\mathbf{h}}-\mathbf{h}\|_0&\leq \|(G+H)(\tilde \rho^2-\rho^2)\|_0+\|\tilde \rho^2(\tilde G-G+\tilde H-H)\|_0\\
&\leq C\delta\lambda^{1-\kappa},\\
\|\tilde {\mathbf{h}}-\mathbf{h}\|_1&\leq C\delta\lambda\,.
\end{split}
\end{equation*}
Thus we conclude
\eqref{e:stage-5}.
\end{proof}


For general dimension $n\geq 3$, we can utilize Lemma \ref{l:decompose} and follow the same argument as in Corollary \ref{c:metric-1} to get a similar result, which is a slightly modified version of Corollary 3.1 in \cite{CS19}. Recall that $r_0=\frac{1}{4}r$, where $r$ is the radius in Lemma \ref{l:decompose}.

\begin{corollary}\label{c:metric-2}
Let $G$ be a $C^1$ metric on $\Omega\subset\R^n$ with $\tfrac{1}{\gamma}\textrm{Id}\leq G\leq \gamma\textrm{Id}$ and $\|G\|_1\leq \gamma$ for some $\gamma\geq 1$. Assume also that $\textrm{osc}_\Omega G<r_0$. Let $u\in C^2(\Omega, \R^{m})$ an immersion, $\rho\in C^1(\Omega)$, $H\in C^1(\Omega;\R^{n\times n}_{sym})$ such that
\begin{align*}
\frac{1}{\gamma}\textrm{Id}\leq \nabla u^T&\nabla u\leq\gamma \textrm{Id}\,\textrm{ in }\Omega,\\
\|u\|_{2}&\leq \delta^{1/2}\lambda,  \\
\|\rho\|_0\leq \delta^{1/2}, &\quad \|\rho\|_1\leq \delta^{1/2}\lambda, \\ 
\|H\|_0\leq \lambda^{-\alpha}, &\quad \|H\|_1\leq \lambda^{1-\alpha}\,,
\end{align*}
for some $0<\delta<1$, $\alpha>0$ and $\lambda>1$ such that $2\gamma\leq r_0\lambda^\alpha$. Then, for any $\kappa\geq 1$
there exists an immersion $v\in C^2(\Omega; \R^{n+1})$ and $\mathcal{E}\in C^1(\Omega;\R^{n\times n}_{sym})$ such that
\begin{equation*}
\begin{split}
 \nabla v^T\nabla v=&\nabla u^T\nabla u+\rho^2(G+H)+\mathcal{E}\quad\textrm{ in }\Omega,\\
v=&u\textrm{ on }\supp\rho+B_{\lambda^{-\kappa}}(0)
\end{split}
\end{equation*}
with estimates
\begin{align}
\|v-u\|_0&\leq C\delta^{1/2}\lambda^{-\kappa}, \label{e:2stage-2}\\
\|v-u\|_1&\leq C\delta^{1/2}, \label{e:2stage-3}\\
\|v\|_2&\leq C\delta^{1/2}\lambda^{n_*(\kappa-1)+1}, \label{e:2stage-4}
\end{align}
and
\begin{equation}\label{e:2stage-5}
\|\mathcal{E}\|_0\leq C\delta\lambda^{1-\kappa}, \quad
\|\mathcal{E}\|_1\leq C\delta\lambda^{(n_*-1)(\kappa-1)+1}.
\end{equation}
Here $C\geq 1$ is a constant depending only on $\gamma$.
\end{corollary}

We point out in passing that in Corollary \ref{c:metric-1} the constant $C$ depended on the exponent $\alpha>0$ through the Schauder estimates in applying Proposition \ref{p:conformal}. Since this step is not available in the higher dimensional setting, eventually the constant $C$ does not depend on $\alpha$. 

\section{Inductive construction of adapted short immersion}\label{s:inductive}

Let $u$ be an adapted short immersion with respect to some compact set $S\subset \mathcal{M}$ with exponent $\theta$ (c.f.~Definition \ref{d:adapt}). In particular
$$
g-u^\sharp e=\rho^2(g+h),
$$
with $S=\{\rho=0\}$. Furthermore, let $\Sigma\supset S$ be another compact subset. Our aim in this section is to show that, under certain conditions, we can modify $u$ to obtain another adapted short immersion with respect to the larger compact set $\Sigma$ with some exponent $\theta'<\theta$. The case we will be interested in is where $\Sigma$ and $S$ are the skeleta of a given triangulation of $\mathcal{M}$ of consecutive dimension - see Section \ref{s:proof-of-theorem} below. In particular both are a finite union of $C^1$ submanifolds on $\mathcal{M}$. In this case the following geometric condition is satisfied: 
\begin{condition}\label{c:geometric}
There exists a geometric constant $\bar{r}>0$ such that
for any $\delta>0$ the set 
$$
\biggl\{x\in\mathcal{M}:\,\textrm{dist}(x,S)\geq \delta\textrm{ and }\textrm{dist}(x,\Sigma)\leq \bar{r}\delta\biggr\}
$$
is contained in a pairwise disjoint union of open sets, each contained in a single chart 
$\Omega_k$.
\end{condition}

We point out explicitly that the special cases where $S=\emptyset$ or $\Sigma=\mathcal{M}$ are admitted in these considerations. 

In the following we focus on the 2-dimensional case, and will briefly point out differences in the analogous argument for $n\geq 3$ in Section \ref{ss:n=3} below.

\subsection{The case $n=2$}\label{ss:n=2}

\begin{proposition}\label{p:inductive}
Let $0<\theta<1/5$ and $0<\alpha<1$. There exists a constant $A_0=A_0(\theta,\alpha)\geq 1$, such that the following holds:

Let $S\subset\Sigma$ be compact subsets of $\mathcal{M}$ satisfying Condition \ref{c:geometric}.  
Let $u\in C^{1, \theta}(\mathcal{ M})$ be an adapted short immersion with $g-u^\sharp e=\rho^2(g+h)$ such that $\rho\leq 1/4$ in $\mathcal{M}$, $S=\{\rho=0\}$, and
in any chart $\Omega_k$
\begin{equation}\label{e:inductive-u-rho-assumption}
\begin{split}
|\nabla^2 u|\leq A\rho^{1-\frac{1}{\theta}}, &\quad |\nabla\rho|\leq A\rho^{1-\frac{1}{\theta}}, \\
|h|\leq A^{-\alpha\theta}\rho^\alpha,&\quad |\nabla h|\leq A^{1-\alpha\theta}\rho^{\alpha-\frac{1}{\theta}},
\end{split}
\end{equation}
for some $A\geq A_0$. Then there exists a new adapted short immersion $\bar{u}\in C^{1, \theta'}(\mathcal{M})$ with $g-\bar{u}^\sharp e=\bar{\rho}^2(g+\bar{h})$ with respect to $\Sigma\supset S$ satisfying $\bar{\rho}\leq\rho$, $\|\bar{u}-u\|_0\leq A^{-1/2}$,  and $\bar{u}=u$ on $S$. Moreover, in any chart $\Omega_k$ 
\begin{equation}\label{e:inductive-conclusion}
\begin{split}
|\nabla^2 \bar{u}|\leq A'\bar{\rho}^{1-\frac{1}{\theta'}}, &\quad |\nabla\bar{\rho}|\leq A'\bar{\rho}^{1-\frac{1}{\theta'}}, \\
|\bar{h}|\leq (A')^{-\alpha'\theta'}\bar{\rho}^{\alpha'},&\quad |\nabla \bar{h}|\leq (A')^{1-\alpha'\theta'}\bar{\rho}^{\alpha'-\frac{1}{\theta'}},
\end{split}
\end{equation}
with
$$A'=A^{b^2}, \quad \theta'=\frac{\theta}{b^2}, \quad  \alpha'=\frac{\alpha}{2b^2},$$
where
\begin{equation}\label{e:theta-b-condition}
b=1+\frac{4\alpha\theta}{1-5\theta}.
\end{equation}

\end{proposition}

The proof of Proposition \ref{p:inductive} is divided into five sections.

\subsection{Parameters}

First, recall from \eqref{e:gamma0} and \eqref{e:gamma1} that for any local chart $\Omega_k$ the coordinate expression $G=(G_{ij})$ of $g$ satisfies
$$
\frac{1}{\gamma_0}\textrm{Id}\leq G\leq \gamma_0\textrm{Id},\quad \|G\|_{C^1(\Omega_k)}\leq \gamma_0,
$$
and let $\gamma:=4\gamma_0$. Since $\rho<1/4$ and $u$ is an adapted (and hence strong) immersion by assumption, using Definition \ref{d:strong} we deduce
$$
\tfrac{1}{4}G\leq (1-\tfrac32\rho^2)g\leq \nabla u^T\nabla u\leq (1-\tfrac{1}{2}\rho^2)G\leq G 
$$
so that 
$$
\frac{1}{\gamma}\textrm{Id}\leq \nabla u^T\nabla u\leq \gamma\textrm{Id}.
$$
Next, set 
\begin{equation}\label{e:deltabar}
\delta_1:=\max_{x\in\mathcal{M}}\rho^2,
\end{equation}
and for $q\geq 1$
\begin{align*}
\lambda_q=A\delta_q^{-\frac{1}{2\theta}},\quad \lambda_{q+1}=\lambda_q^b.
\end{align*}
By ensuring that $A$ is sufficiently large (depending on $\theta,\alpha$), we may then assume without loss of generality that
\begin{equation}\label{e:ordering}
\delta_{q+1}\leq \frac{1}{4}\delta_q,\quad \lambda_{q+1}\geq 2\lambda_q.
\end{equation}
\subsection{Definition of cut-off functions}\label{s:cut-off}

We first decompose $\mathcal{M}$ with respect to $\Sigma$ and $S$. Let  
$$
r_q=A^{-1}\delta_{q+1}^{\tfrac{1}{2\theta}}=\lambda_{q+1}^{-1},
$$
and define for $q=0, 1,2,\dots$
\begin{align*}
\Sigma_q&=\{x:\,\textrm{dist}(x,\Sigma)<r_*r_q\},\\
\widetilde\Sigma_q&=\{x:\,\textrm{dist}(x,\Sigma)<\tilde r_*r_q\},\\
S_q&=\{x:\,\textrm{dist}(x,S)<r_{**}r_q\},
\end{align*}
where $r_*<\tilde r_{*}$ and $r_{**}$ are geometric constants to be chosen as follows:
First of all, choose $r_{**}>0$ so that
\begin{equation}\label{e:choiceofr**}
\rho(x)>\tfrac{3}{2}\delta_{q+2}^{1/2}\quad\textrm{ implies }\quad x\notin S_{q+1}.
\end{equation}
Indeed, recall from \eqref{e:adapt-final-2} and the discussion following it, that $\rho$ is $\theta$-H\"older continuous and hence 
$\rho(x)\leq A^{\theta}\textrm{dist}(x,S)^\theta$. In particular, for any $x\in S_{q+1},$ one has $\rho(x)\leq r_{**}^\theta\delta_{q+2}^{1/2}$. Thus such a choice of $r_{**}$ is possible. 

Then, set $\tilde r_*=\bar{r}r_{**}$, where $\bar{r}>0$ is the geometric constant in Condition \ref{c:geometric}; this ensures that, for any $q\in\N$
\begin{equation}\label{e:singlecharts}
\begin{split}
	\widetilde\Sigma_q\setminus S_q&\textrm{ \emph{is contained in a pairwise disjoint union of open sets,}}\\
	&\textrm{\emph{each contained in a single chart }}\Omega_k	
\end{split}
\end{equation}
Finally, choose $r_*<\tilde r_*$ so that $\tfrac{1}{2}\tilde r_*<r_*<\tilde r_*$, hence (in light of \eqref{e:ordering})
$$
\widetilde\Sigma_{q+1}\subset \Sigma_q\subset \widetilde\Sigma_q\quad\textrm{ for all }q.
$$  
Next, we fix cut-off functions $\phi, \tilde\phi, \psi, \tilde\psi\in C^\infty(0,\infty)$ with $\phi,\tilde\phi$ monotonic increasing, $\psi,\tilde\psi$ monotonic decreasing such that
$$
\phi(s),\tilde\phi(s)=\begin{cases} 1&s\geq 2\\ 0&s\leq \tfrac32\end{cases}\,,\quad  
\psi(s),\tilde\psi(s)=\begin{cases} 1&s\leq r_*\\ 0&s\geq \tilde r_*\end{cases}\,,
$$
and in addition
$$
\tilde\phi(s)=1 \textrm{ on }\supp\phi,\quad \tilde\psi(s)=1 \textrm{ on }\supp\psi.
$$
Set
\begin{align*}
\chi_q(x)=\phi\left(\frac{\rho(x)}{\delta_{q+2}^{1/2}}\right)\psi\left(\frac{\textrm{dist}(x, \Sigma)}{r_{q+1}}\right),\quad
\tilde\chi_q(x)=\tilde\phi\left(\frac{\rho(x)}{\delta_{q+2}^{1/2}}\right)\tilde\psi\left(\frac{\textrm{dist}(x, \Sigma)}{r_{q+1}}\right).
\end{align*}
Using \eqref{e:inductive-u-rho-assumption} and the choice of $r_q$, $r_*$, $\tilde r_*$ and the cut-off functions we easily deduce
 \begin{align}
 |\nabla\chi_q|,\, |\nabla\tilde\chi_q|&\leq CA\delta_{q+2}^{-\frac{1}{2\theta}}=C\lambda_{q+2},\label{e:gradient-chi}\\
\textrm{dist}(\supp\chi_q, \partial\supp\tilde\chi_q)&\geq C^{-1}A^{-1}\delta_{q+2}^{\frac{1}{2\theta}}=C^{-1}\lambda^{-1}_{q+2}.\label{e:chi-q-support}
 \end{align}
for some constant $C$ depending on $r_*,\tilde r_*$, and moreover
\begin{equation}\label{e:chi-define}
\begin{split}
\{x\in\Sigma_{q+1}|\rho(x)>2\delta_{q+2}^{1/2}\}&\subset\{x\in\mathcal{M}:\,\chi_q(x)=1\},\\
\supp\chi_q&\subset \{x\in\mathcal{M}:\,\tilde{\chi}_q(x)=1\},\\
\supp\tilde{\chi}_q&	\subset\{x\in\widetilde\Sigma_{q+1}:\,\rho(x)>\tfrac{3}{2}\delta_{q+2}^{1/2}\}.
\end{split}
\end{equation}
From \eqref{e:choiceofr**}  and \eqref{e:singlecharts} we then deduce that $\supp\tilde\chi_q$ is contained in a pairwise disjoint union of open sets, each contained in a single chart $\Omega_k$. 

\subsection{Construction of error size sequence}
Our strategy of proving Proposition \ref{p:inductive} is constructing a sequence of adapted short immersions, which is based on the above cut-off functions and error size. Thus we first define  the sequence of error size $\{\rho_q\}.$
Set $\rho_0=\rho$ and define $\rho_{q}$ for $q=1,2,\dots$ inductively as
\begin{equation}\label{e:rho-q+1}
\rho_{q+1}^2=\rho_q^2(1-\chi_q^2)+\delta_{q+2}\chi_q^2.
\end{equation}

\begin{lemma}\label{l:rho}
Let $\{\rho_q\}$ be defined in \eqref{e:rho-q+1}. Then for any $q=0,1,\dots$ 
\begin{enumerate}
	\item[(i)] On $\supp\tilde\chi_q$ it holds
	\begin{equation*}
	\tfrac{3}{2}\delta_{q+2}^{1/2}\leq\rho_q\leq2\delta_{q+1}^{1/2}.
	\end{equation*}
	\item[(ii)] For every $x$ we have $\rho_{q+1}(x)\leq\rho_q(x)$.
	\item[(iii)] If $\rho_q(x)\leq\delta_{q+1}^{1/2}$, then $x\not\in\bigcup_{j=0}^{q-1}\supp\tilde\chi_j$ and consequently $\rho_q(x)=\rho(x)$.
	\item[(iv)] If $\rho_q(x)\geq \delta_{q+1}^{1/2}$, then either $\chi_q(x)=1$ or $x\notin \Sigma_{q+1}$.
	\end{enumerate}
\end{lemma}
\begin{proof} 
\hfill

\noindent{\bf{(i)} }
The statement (i) holds for the case $q=0$ by the property \eqref{e:chi-define} of $\chi_0$ and the definition of $\delta_1$, i.e., \eqref{e:deltabar}. 
We now proceed by induction and assume that (i) holds for $\rho_q$, and let $x\in \supp\tilde\chi_{q+1}$. Observe that in particular $x\in \Sigma_{q+1}$. We consider two cases.

\textit{Case 1.} If $x\notin\supp\chi_q$, then from \eqref{e:chi-define} it follows that
$$
\tfrac32\delta_{q+3}^{1/2}\leq\rho_0(x)\leq 2\delta_{q+2}^{1/2}.
$$
Furthermore, since $x\in \Sigma_{q+1}\subset\tilde\Sigma_{q+1}$, using the definition of $\tilde\chi_q$ we deduce that necessarily $\delta_{q+2}^{-1/2}\rho(x)\notin \supp\tilde\phi$. Consequently $x\notin\supp\chi_j$ for all $j\leq q$, and hence $\rho_{q+1}(x)=\rho_0(x)$. This concludes (i) in this case.  

\textit{Case 2.} If $x\in\supp\chi_{q}\subset\supp\tilde\chi_q$, then from the induction hypothesis we have $\rho_q(x)\geq\tfrac32\delta_{q+2}^{1/2}$. Hence
$$
\rho_{q+1}(x)\geq\min(\rho_q(x),\delta_{q+2}^{1/2})= \delta_{q+2}^{1/2} \geq\tfrac32\delta_{q+3}^{1/2}.
$$
For the upper bound, either $\chi_q(x)=1$ in which case $\rho_{q+1}(x)=\delta_{q+2}^{1/2}$, or $0<\chi_q(x)<1$, implying by \eqref{e:chi-define} (and since $x\in \Sigma_{q+1}$) that $\rho_0(x)\leq 2\delta_{q+2}^{1/2}$. As above, we deduce that in this case $x\notin\supp\chi_j$ for all $j\leq q-1$ and $\rho_{q}(x)=\rho_0(x)$. It follows that $\rho_{q}(x)\leq 2\delta_{q+2}^{1/2}$, hence
$$
\rho_{q+1}(x)\leq\max(\rho_q(x),\delta_{q+2}^{1/2})\leq 2\delta_{q+2}^{1/2}.
$$
In either case we conclude (i).

\medskip

\noindent{\bf{(ii)} }
Note that from the definition \eqref{e:rho-q+1} it follows
$$
\rho_{q+1}(x)\leq\max(\rho_q(x),\delta_{q+2}^{1/2}).
$$
Moreover, if $\rho_q(x)\leq \delta_{q+2}^{1/2}$, then (i) implies that $\chi_q(x)=0$ and consequently $\rho_{q+1}(x)=\rho_q(x)$. The conclusion (ii) easily follows.

\medskip

\noindent{\bf{(iii) }}
Assume that $\rho_q(x)\leq \delta_{q+1}^{1/2}$. From the expression \eqref{e:rho-q+1} we deduce that $\rho_{q-1}(x)\leq \delta_{q+1}^{1/2}$, and by the same reasoning we further deduce recursively that $\rho_j(x)\leq \delta_{q+1}^{1/2}$ for all $j\leq q$. From (i) the conclusion (iii) then follows.

\medskip

\noindent{\bf{(iv) }}
Assume that $\rho_q(x)\geq \delta_{q+1}^{1/2}$ and $x\in \Sigma_{q+1}$. Then, using (ii) we deduce $\rho(x)\geq \delta_{q+1}^{1/2}>2\delta_{q+2}^{1/2}$, hence from \eqref{e:chi-define} $\chi_q(x)=1$.

\end{proof}

%

Now we are ready to inductively construct  a sequence of adapted short immersions.

\subsection{Inductive construction} 
We will construct a sequence of smooth
adapted short immersions $(u_q, \rho_q, h_q)$ such that the following hold:
\begin{itemize}
\item[$(1)_q$] For all $\mathcal{M},$ we have
$$g-u_q^\sharp e=\rho_q^2(g+h_q).$$
\item[$(2)_q$] If $x\notin\bigcup_{j=0}^{q-1}\supp\tilde\chi_j$, then $(u_q, \rho_q, h_q)=(u_0, \rho_0, h_0).$
\item[$(3)_q$] The following estimates hold in $\mathcal{M}$:
\begin{align}
|\nabla^2u_q|\leq A^{b^2}\rho_q^{1-\frac{b^2}{\theta}},\quad & |\nabla\rho_q|\leq A^{b^2}\rho_q^{1-\frac{b^2}{\theta}} ,\label{e:inductive-v-rho-j}\\
|h_q|\leq A^{-\frac{\theta\alpha}{2b^2}}\rho_q^{\frac{\alpha}{2b^2}},\quad &|\nabla h_q|\leq A^{{b^2}-\frac{\theta\alpha}{2b^2}}\rho_q^{\frac{\alpha}{2b^2}-\frac{b^2}{\theta}},
\label{e:inductive-h-j}
\end{align}
\item[$(4)_q$] On $\{x: \rho_0(x)>\delta_{q+1}^{1/2}\}\cap\Sigma_q,$  we have the sharper estimates
\begin{align}
|\nabla^2u_q|\leq A^{b}\rho_q^{1-\frac{b}{\theta}}, \quad &|\nabla\rho_q|\leq A^b\rho_q^{1-\frac{b}{\theta}} ,\label{e:inductive-rho-q}\\
|h_q|\leq A^{-\frac{\theta\alpha}{b}}\rho_q^{\frac{\alpha}{b}}, \quad &|\nabla h_q|\leq A^{b-\frac{\theta\alpha}{b}}\rho_q^{\frac{\alpha}{b}-\frac{b}{\theta}}.
\label{e:inductive-h-q}
\end{align}
\item[$(5)_q$] We have the global estimate for $q\geq1$
\begin{align}
&\|u_q-u_{q-1}\|_0\leq \overline{C}\delta_q^{1/2}\lambda_{q}^{-1},\label{e:inductive-v-difference-0}\\
&\|u_q-u_{q-1}\|_1\leq\overline{C}\delta_q^{1/2},\label{e:inductive-v-difference-1}
\end{align}
where $\overline{C}$ is the constant in the conclusions of Corollary \ref{c:metric-1} in \eqref{e:stage-2}-\eqref{e:stage-3}.
\end{itemize}

%

\subsection*{Initial step $q=0$} Set $(u_0, \rho_0, h_0)=(u, \rho, h).$ Since $b>1,$ it is easy to check $(1)_0-(2)_0$ and $(4)_0$ from \eqref{e:inductive-u-rho-assumption}.

\subsection*{Inductive step $q\mapsto q+1$} Suppose $(u_q, \rho_q, h_q)$  is an adapted short immersion on $\mathcal{M}$ satisfying $(1)_q-(5)_q.$ We then construct $(u_{q+1}, \rho_{q+1}, h_{q+1}).$ In fact, $\rho_{q+1}$ has already been defined in \eqref{e:rho-q+1}. 
We shall estimate $(u_q, \rho_q, h_q)$ on $\supp\tilde\chi_q.$  By (i) in Lemma \ref{l:rho}, on $\supp\tilde\chi_q,$
\begin{equation}\label{e:rho-q-bound}
\tfrac{3}{2}\delta_{q+2}^{1/2}\leq\rho_q\leq2\delta_{q+1}^{1/2}.
\end{equation}
If $\frac{3}{2}\delta_{q+2}^{1/2}\leq\rho_q(x)\leq\delta_{q+1}^{1/2},$ then using Lemma \ref{l:rho} (iii) we infer that  $x\not\in\cup_{j=0}^{q-1}\supp\tilde\chi_q$ so that $(u_q, \rho_q, h_q)=(u_0, \rho_0, h_0)$. From \eqref{e:inductive-u-rho-assumption} one has
\begin{equation}\label{e:rho-q-less}
\begin{split}
|\nabla^2u_q|&=|\nabla^2 u_0|\leq A\rho_0^{1-\frac{1}{\theta}}\leq A\delta_{q+2}^{\frac{1}{2}(1-\frac{1}{\theta})}= \delta_{q+2}^{1/2}\lambda_{q+2},\\
|\nabla\rho_q|&=|\nabla\rho_0|\leq A\rho_0^{1-\frac{1}{\theta}}\leq A\delta_{q+2}^{\frac{1}{2}(1-\frac{1}{\theta})}= \delta_{q+2}^{1/2}\lambda_{q+2},\\
\left|\frac{\nabla\rho_q}{\rho_q}\right|&=\left|\frac{\nabla\rho_0}{\rho_0}\right|\leq A\rho_0^{-\frac{1}{\theta}}\leq A\delta_{q+2}^{-\frac{1}{2\theta}}= \lambda_{q+2},\\
|h_q|&=|h_0|\leq A^{-\theta\alpha}\rho_0^{\alpha}\leq A^{-\theta\alpha}\delta_{q+1}^{\alpha/2}=
\lambda_{q+1}^{-\theta\alpha},\\
|\nabla h_q|&=|\nabla h_0|\leq A^{1-\theta\alpha}\rho_0^{\alpha-\frac{1}{\theta}}\leq  A^{1-\theta\alpha}\delta_{q+2}^{\frac{1}{2}(\alpha-\frac{1}{\theta})}= \lambda_{q+2}^{1-\alpha\theta}.
\end{split}
\end{equation}
On the other hand, if $\delta_{q+1}^{1/2}<\rho_q(x)\leq2\delta_{q+1}^{1/2},$ then Lemma \ref{l:rho} (ii) implies $\rho_0(x)\geq\rho_q(x)>\delta_{q+1}^{1/2}$. Therefore \eqref{e:inductive-rho-q}-\eqref{e:inductive-h-q} in $(4)_q$ leads to 
\begin{equation}\label{e:rho-q-large}
\begin{split}
|\nabla^2u_q|&\leq A^b\rho_q^{1-\frac{b}{\theta}}\leq A^b\delta_{q+1}^{\frac{1}{2}(1-\frac{b}{\theta})}= \delta_{q+1}^{1/2}\lambda_{q+2},\\
|\nabla\rho_q|&\leq A^b\rho_q^{1-\frac{b}{\theta}}\leq A^b\delta_{q+1}^{\frac{1}{2}(1-\frac{b}{\theta})}= \delta_{q+1}^{1/2}\lambda_{q+2},\\
\left|\frac{\nabla\rho_q}{\rho_q}\right|&\leq A^b\rho_q^{-\frac{b}{\theta}}\leq A^b\delta_{q+1}^{-\frac{b}{2\theta}}= \lambda_{q+2},\\
|h_q|&\leq A^{-\frac{\theta\alpha}{b}}\rho_q^{\frac{\alpha}{b}}\leq 2A^{-\frac{\theta\alpha}{b}}\delta_{q+1}^{\frac{\alpha}{2b}}=2\lambda_{q+1}^{-\frac{\alpha\theta}{b}}, \\
|\nabla h_q|&\leq A^{b-\frac{\theta\alpha}{b}}\rho_q^{\frac\alpha b-\frac{b}{\theta}}\leq A^{b-\frac{\theta\alpha}{b}}\delta_{q+1}^{\frac{1}{2}(\frac{\alpha}{b}-\frac{b}{\theta})}=\lambda_{q+1}^{b-\frac{\theta\alpha}{b}}.
\end{split}
\end{equation}
Combining \eqref{e:rho-q-less} with \eqref{e:rho-q-large}, using \eqref{e:inductive-rho-q}-\eqref{e:inductive-h-q}, one finally has on $\supp\tilde\chi_q,$
\begin{equation}\label{e:support-rho_q}
\begin{split}
\tfrac{3}{2}\delta_{q+2}^{1/2}&\leq \rho_q\leq2\delta_{q+1}^{1/2},\\
 |\nabla \rho_q|&\leq \delta_{q+1}^{1/2}\lambda_{q+2} ,\, |\nabla^2u_q|\leq \delta_{q+1}^{1/2}\lambda_{q+2},\,\left|\frac{\nabla\rho_q}{\rho_q}\right|\leq \lambda_{q+2}\\
|h_q|&\leq 2\lambda_{q+2}^{-\frac{\theta\alpha}{b^2}}, \,\, |\nabla h_q|\leq \lambda_{q+2}^{1-\frac{\theta\alpha}{b^2}}.
\end{split}
\end{equation}
We then apply Corollary \ref{c:metric-1} to construct $(u_{q+1}, h_{q+1}).$ To this end define
\begin{align*}
\tilde{\rho}_q=\chi_q\sqrt{\rho_q^2-\delta_{q+2}},\quad \tilde h_q=\frac{\tilde\chi_q\rho_q^2}{\rho_q^2-\delta_{q+2}}h_q,
\end{align*}
then
$$\tilde\rho_q^2(g+\tilde h_q)=\chi_q^2(\rho_q^2(g+h_q)-\delta_{q+2}g)=\chi_q^2(g-u_{q}^\sharp e-\delta_{q+2}g).$$
From \eqref{e:rho-q-bound}, one has on $\supp\tilde\chi_q$
$$\tfrac54\delta_{q+2}\leq\rho_q^2-\delta_{q+2}\leq4\delta_{q+1},$$
hence $\tilde\rho_q$ and $ \tilde h_q$ are well defined. Besides, on $\supp\tilde\chi_q,$ we have
 \begin{align*}
 |\nabla\sqrt{\rho_q^2-\delta_{q+2}}|&\leq C|\nabla\rho_q|,\\
\frac{\rho_q^2}{\rho_q^2-\delta_{q+2}}&=1+\frac{\delta_{q+2}}{\rho_q^2-\delta_{q+2}}\leq2,\\
\left|\nabla\frac{\rho_q^2}{\rho_q^2-\delta_{q+2}}\right|&=\left|\nabla\frac{\delta_{q+2}}{\rho_q^2-\delta_{q+2}}\right|
\leq C\left|\frac{\nabla\rho_q}{\rho_q}\right|,
 \end{align*}
where $C$ are geometric constants. Therefore, using \eqref{e:gradient-chi} and \eqref{e:support-rho_q} we can infer
\begin{equation}\label{e:rho-tilde}
\begin{split}
0\leq\tilde\rho_q&\leq\rho_q\leq2\delta_{q+1}^{1/2},\\
|\nabla\tilde\rho_q|&\leq C(|\nabla\chi_q|\rho_q+|\nabla\rho_q|)\leq C\delta_{q+1}^{1/2}\lambda_{q+2},\\
|\tilde h_q|&\leq2|h_q|\leq4\lambda_{q+2}^{-\frac{\theta\alpha}{b^2}},\\
|\nabla\tilde h_q|&\leq C(|\nabla\tilde\chi_q||h_q|+\left|\frac{\nabla\rho_q}{\rho_q}\right||h_q|+|\nabla h_q|)\leq C\lambda_{q+2}^{1-\frac{\theta\alpha}{b^2}}.
\end{split}
\end{equation}
Therefore $(u_q, \tilde\rho_q, \tilde h_q)$ satisfies all the assumptions in Corollary \ref{c:metric-1} on $\supp\tilde\chi_q$ with $\delta, \lambda, \alpha$ given by $4\delta_{q+1}, C\lambda_{q+2}, \frac{\theta\alpha}{4b^2}$ respectively. We recall \eqref{e:singlecharts} that  $\supp\tilde\chi_q$ is contained in a pairwise disjoint union of open sets, each contained in a single chart. Therefore,  we may apply Corollary \ref{c:metric-1} in each open set separately in local coordinates to add the term $\tilde\rho_q^2(g+\tilde h_q)$ and take $\kappa$ in Corollary \ref{c:metric-1} as 
$$
\kappa=1+\frac{2\theta}{b}(b-1+\alpha)>1.
$$
Overall we obtain $u_{q+1}$ and $\mathcal{E}$ such that
$$g-u_{q+1}^\sharp e=(g-u_q^\sharp e)(1-\chi_q^2)+\delta_{q+2}g\chi_q^2+\mathcal{E}.$$
with $u_{q+1}$ satisfying
\begin{equation}\label{e:v-q+1-c2}
|\nabla^2u_{q+1}|\leq C\delta_{q+1}^{1/2}\lambda_{q+2}^{2\kappa-1}=C\delta_{q+1}^{1/2}\lambda_{q+1}^{b+4\theta(b-1)+4\theta\alpha},
\end{equation}
and $\mathcal{E}$ satisfying
\begin{align}
&|\mathcal{E}|\leq C\delta_{q+1}\lambda_{q+2}^{1-\kappa}=C\delta_{q+2}\lambda_{q+1}^{-2\theta\alpha}, \label{e:error-q+1--1}\\
&|\nabla\mathcal{E}|\leq C\delta_{q+1}\lambda_{q+2}^{\kappa}=C\delta_{q+2}\lambda_{q+1}^{b+{4\theta}(b-1)+2\theta\alpha}, \label{e:error-q+1--2}
\end{align}
which are implied by \eqref{e:stage-4}-\eqref{e:stage-5}. From  \eqref{e:error-support}, one gets
$$\supp(u_{q+1}-u_q), \,\supp\mathcal{E}\subset\supp\chi_q+B_{\tau_q}(0),$$
with
$$
\tau_q=(C\lambda_{q+2})^{-\kappa}\leq A^{-{2\theta}(b-1+\alpha)}\lambda_{q+2}^{-1}\leq C^{-1}\lambda_{q+2}^{-1},
$$
where $C$ is the constant in \eqref{e:chi-q-support} and the last inequality holds provided $A$ is sufficiently large. Consequently $u_{q+1}=u_q$ and $\mathcal{E}=0$ outside  $\supp\tilde\chi_q$.

Moreover, \eqref{e:inductive-v-difference-0} and \eqref{e:inductive-v-difference-1} for the case $q+1$ follow immediately from \eqref{e:stage-2}-\eqref{e:stage-3}, hence $(5)_{q+1}$ is verified. We also define
$$h_{q+1}=(1-\chi_q^2)\frac{\rho_q^2}{\rho_{q+1}^2}h_q+\frac{\mathcal{E}}{\rho_{q+1}^2}$$
so that $$g-u_{q+1}^\sharp e=\rho_{q+1}^2(g+h_{q+1}),$$
verifying $(1)_{q+1}.$  Note that on $\supp\tilde\chi_q$ using \eqref{e:rho-q-bound} one has
\begin{equation}\label{e:rhoq+1-bound}
\begin{split}
\rho_{q+1}^2&\leq4\delta_{q+1}(1-\chi_q^2)+\delta_{q+2}\chi_q^2\leq 4\delta_{q+1},\\
\rho_{q+1}^2&\geq\tfrac94\delta_{q+2}(1-\chi_q^2)+\delta_{q+2}\chi_q^2\geq\delta_{q+2}.
\end{split}
\end{equation}
Thus both $\mathcal{E}$ and $h_{q+1}$ are well defined. Besides we can also derive that $(\rho_{q+1}, h_{q+1})$ agrees with $(\rho_q, h_q)$ outside $\supp\tilde\chi_q.$ It remains to verify $(2)_{q+1}-(4)_{q+1}$ on $\supp\tilde\chi_q.$

\subsubsection*{Verification of $(2)_{q+1}$} If $x\not\in\bigcup_{j=0}^{q}\supp\tilde\chi_j$, then  $\tilde\chi_q(x)=0$ and therefore
$$(u_{q+1}, \rho_{q+1}, h_{q+1})=(u_q, \rho_q, h_q)=(u_0, \rho_0, h_0).$$

\subsubsection*{Verification of $(3)_{q+1}$}  On $\supp\tilde\chi_q,$ we first calculate 
\begin{equation}\label{e:gradient-rho-q+1}
\begin{split}
|\nabla\rho_{q+1}|&=\frac{|\nabla\rho_q^2|}{2\rho_{q+1}}\leq\frac{C}{\rho_{q+1}}(|\rho_q\nabla\rho_q|
+|\nabla\chi_q|(\rho_q^2+\delta_{q+2}))\\
&\leq C\frac{\delta_{q+1}\lambda_{q+2}}{\delta_{q+2}^{1/2}}=CA^{b+(b-1)\theta}\delta_{q+1}^{1-\frac{b}{2}(1+\frac{1}{\theta})}\\
&\leq A^{b^2}(2\delta_{q+1}^{1/2})^{1-\frac{b^2}{\theta}}\leq A^{b^2}\rho_{q+1}^{1-\frac{b^2}{\theta}},
\end{split}
\end{equation}
where we have used \eqref{e:gradient-chi}, \eqref{e:support-rho_q} and \eqref{e:rhoq+1-bound}. For the inequality in the last line we have used that $1-\frac{b}{2}(1+\frac{1}{\theta})\geq \frac{1}{2}(1-\frac{b^2}{\theta})$, $5(b-1)\theta+b\leq b^2$ (from \eqref{e:theta-b-condition}) and $A$ sufficiently large to absorb geometric constants.  

Similarly, using \eqref{e:support-rho_q}, \eqref{e:v-q+1-c2}-\eqref{e:error-q+1--1} and \eqref{e:rhoq+1-bound} we obtain
\begin{equation}\label{e:hq+1-bound-v-c2}
\begin{split}
|h_{q+1}|&\leq|h_q|+\frac{|\mathcal{E}|}{\rho_{q+1}^2}\leq 2\lambda_{q+2}^{-\frac{\alpha\theta} {b^2}}+C\lambda_{q+1}^{-2\theta\alpha}\leq CA^{-\frac{\alpha\theta}{b^2}}\delta_{q+2}^{\frac{\alpha}{2b^2}}\\
&\leq A^{\frac{-\alpha\theta}{2b^2}}\rho_{q+1}^{\frac{\alpha}{2b^2}}\,,\\
|\nabla^2u_{q+1}|&\leq C\delta_{q+1}^{1/2}\lambda_{q+1}^{b+4\theta(b-1)+4\alpha\theta}\leq C\delta_{q+1}^{1/2}\lambda_{q+1}^{b^2-\theta(b-1)}\leq CA^{b^2-\theta(b-1)}\delta_{q+1}^{\frac{1}{2}(1-\frac{b^2}{\theta})}\\
&\leq A^{b^2}\rho_{q+1}^{1-\frac{b^2}{\theta}},
\end{split}
\end{equation}
where we have used $5\theta(b-1)+4\alpha\theta\leq b^2-b$ (c.f.~\eqref{e:theta-b-condition}) and again assumed $A$ sufficiently large to absorb the constants $C$.
For $|\nabla h_{q+1}|,$ we calculate as follows.
\begin{equation}\label{e:gradient-h-q+1}
\begin{split}
|\nabla h_{q+1}|&\leq|\nabla h_q|+\frac{1}{\rho_{q+1}^2}(|\nabla\mathcal{E}|+\delta_{q+2}|\nabla(h_q\chi_q^2)|)+\frac{2|\nabla\rho_{q+1}|}{\rho_{q+1}^3}(\delta_{q+2}|h_q|+|\mathcal{E}|)\\
&\leq C\lambda_{q+2}^{1-\frac{\theta\alpha}{b^2}}
+C(\lambda_{q+1}^{b+4\theta(b-1)+2\theta\alpha}+\lambda_{q+2}^{1-\frac{\theta\alpha}{b^2}})+C\frac{\delta_{q+1}\lambda_{q+2}}{\delta_{q+2}}
(\lambda_{q+2}^{\frac{-\theta\alpha}{b^2}}+\lambda_{q+1}^{-2\theta\alpha})\\
&\leq C\lambda_{q+2}^{1-\frac{\theta\alpha}{b^2}}
+C\lambda_{q+1}^{b+4\theta(b-1)+2\theta\alpha}+C\frac{\delta_{q+1}}{\delta_{q+2}}
\lambda_{q+2}^{1-\frac{\theta\alpha}{b^2}}\\
&\leq C\lambda_{q+1}^{b+4\theta(b-1)+2\theta\alpha},
\end{split}
\end{equation}
where we have used \eqref{e:gradient-chi}, \eqref{e:support-rho_q}, \eqref{e:error-q+1--1}, \eqref{e:error-q+1--2} and \eqref{e:gradient-rho-q+1}.
Using again the inequality $5\theta(b-1)+4\alpha\theta\leq b^2-b$,  we further estimate
\begin{equation}\label{e:gradient-h-q+1-0}
\begin{split}
|\nabla h_{q+1}|&\leq C\lambda_{q+1}^{b^2-2\alpha\theta-\theta(b-1)}= CA^{b^2-2\alpha\theta-\theta(b-1)}\delta_{q+1}^{\alpha-\frac{b^2}{2\theta}}\\
&\leq A^{b^2-2\alpha\theta}\rho_{q+1}^{2\alpha-\frac{b^2}{\theta}},
\end{split}
\end{equation}
where we have again used that $A$ is sufficiently large. Thus we have shown \eqref{e:inductive-v-rho-j} for $q+1,$ i.e. $(3)_{q+1}$ is verified.

\subsubsection*{Verification of $(4)_{q+1}$} Observe that
\begin{align*}
\{x\in\Sigma_{q+1}: \rho_0(x)>\delta_{q+2}^{1/2}\}
=\{\chi_q(x)=1\}\cup\{x\in\Sigma_{q+1}: \delta_{q+2}^{1/2}\leq\rho_0(x)\leq2\delta_{q+2}^{1/2}\}.
\end{align*}
If $x\in\{\chi_q=1\},$ then
$$\rho_{q+1}=\delta_{q+2}^{1/2},\quad  h_{q+1}=\frac{\mathcal{E}}{\delta_{q+2}}.$$
Using \eqref{e:v-q+1-c2}, 
\begin{equation}\label{e:vq+1}
|\nabla^2 u_{q+1}|\leq C\delta_{q+1}^{1/2}\lambda_{q+1}^{b+4\theta(b-1)+4\alpha\theta}\leq C\delta_{q+1}^{1/2}\lambda_{q+1}^{2b-\theta(b-1)-1}\leq CA^{2-\frac{1}{b}-b}A^b\delta_{q+2}^{\frac{1}{2}(1-\frac{b}{\theta})},
\end{equation}
where we have used $5\theta(b-1)+4\alpha\theta=b-1$ (c.f.~\eqref{e:theta-b-condition}). Since $2-\frac{1}{b}<b$, by taking $A$ sufficiently large we absorb the geometric constant $C$ and deduce \eqref{e:inductive-rho-q}. 

In order to verify \eqref{e:inductive-h-q} we calculate using \eqref{e:error-q+1--1}-\eqref{e:error-q+1--2}:
\begin{equation*}
\begin{split}
|h_{q+1}|&\leq C\lambda_{q+2}^{-\frac{2\alpha\theta}{b}}=CA^{-\frac{2\alpha\theta}{b}}\delta_{q+2}^{\frac{\alpha}{b}},\\
|\nabla h_{q+1}|&\leq C\lambda_{q+1}^{b+4\theta(b-1)+2\alpha\theta}\leq C\lambda_{q+2}^{b-\frac{2\alpha\theta}{b}}=CA^{b-\frac{2\alpha\theta}{b}}\delta_{q+2}^{\frac{\alpha}{b}-\frac{b}{2\theta}}
\end{split}
\end{equation*}
using once again \eqref{e:theta-b-condition}. By choosing $A$ sufficiently large, we can then absorb again the geometric constants and conclude  \eqref{e:inductive-h-q}. Hence $(4)_{q+1}$ is obtained for this case.

\smallskip

On the other hand, if $x\in\{x\in\Sigma_{q+1}: \delta_{q+2}^{1/2}\leq\rho_0(x)\leq2\delta_{q+2}^{1/2}\},$
then $(u_q, \rho_q, h_q)=(u_0, \rho_0, h_0)$ by $(2)_q$ and $\rho_0\leq2\delta_{q+2}^{1/2}$.
Thus
\begin{equation}\label{e:rho-q+1-bound}
\begin{split}
\rho_{q+1}^2&\geq\delta_{q+2}(1-\chi_q^2)+\delta_{q+2}\chi_q^2\geq\delta_{q+2},\\
\rho_{q+1}^2&\leq4\delta_{q+2}(1-\chi_q^2)+\delta_{q+2}\chi_q^2\leq4\delta_{q+2}.
\end{split}
\end{equation}
Therefore, choosing again $A$ sufficiently large to absorb geometric constants,  
\begin{equation}
\label{e:h-q+1-bound}
\begin{split}
|h_{q+1}|&\leq |h_0|+\left|\frac{\mathcal{E}}{\rho_{q+1}^2}\right|\\
&\leq A^{-\alpha\theta}\delta_{q+2}^{\alpha/2}+C\lambda_{q+2}^{-\frac{2\alpha\theta}{b}}\\
&\leq A^{-\frac{\alpha\theta}{b}}\delta_{q+2}^{\frac{\alpha}{2b}}\leq A^{-\frac{\alpha\theta}{b}}\rho_{q+1}^{\frac{\alpha}{b}}.
\end{split}
\end{equation}
Moreover, calculating as in \eqref{e:gradient-rho-q+1} but this time using \eqref{e:rho-q+1-bound}
\begin{equation*}
\begin{split}
|\nabla\rho_{q+1}|&=\frac{|\nabla\rho_q^2|}{2\rho_{q+1}}\leq\frac{C}{\rho_{q+1}}(|\rho_q\nabla\rho_q|
+|\nabla\chi_q|(\rho_q^2+\delta_{q+2}))\\
&\leq C\delta_{q+2}^{1/2}\lambda_{q+2}=CA\delta_{q+2}^{\frac{1}{2}(1-\frac{1}{\theta})}\\
&\leq A^{b}\delta_{q+2}^{\frac{1}{2}(1-\frac{b}{\theta})}\leq A^{b}\rho_{q+1}^{1-\frac{b}{\theta}}.
\end{split}
\end{equation*}
Similarly, proceeding as in \eqref{e:gradient-h-q+1}-\eqref{e:gradient-h-q+1-0} we have
\begin{equation*}
|\nabla h_{q+1}|\leq C\lambda_{q+2}^{b-\frac{2\alpha\theta}{b}-\theta(1-\frac{1}{b})}=CA^{b-\frac{2\alpha\theta}{b}-\theta(1-\frac{1}{b})}\delta_{q+2}^{\frac{\alpha}{b}-\frac{b}{2\theta}+\frac{1}{2}(1-\frac{1}{b})}\leq A^{b-\frac{\theta\alpha}{b}}\rho_{q+1}^{\frac\alpha b-\frac{b}{\theta}}.
\end{equation*}
Finally, the estimate for $\nabla^2u_{q+1}$ has already been obtained in \eqref{e:vq+1}. 
Therefore $(4)_{q+1}$ is verified also in this case.

Overall we have shown that $(u_{q+1}, \rho_{q+1}, h_{q+1})$ satisfies $(1)_{q+1}-(5)_{q+1}.$

\subsection{Conclusion}\label{ss:conlude}

We are now in a position to take the limit as $q\rightarrow\infty.$ Recalling \eqref{e:ordering} we see that $\delta_q^{1/2}\leq 2^{-q-1}$ and $\delta_q^{1/2}\lambda_q^{-1}\leq A^{-1}2^{-q-1}$. In particular from $(5)_q$ we see that $\{u_q\}$ is a Cauchy sequence in $C^1(\mathcal{M})$. 

From the formula \eqref{e:rho-q+1} and Lemma \ref{l:rho} we deduce $0\leq \rho_{q}-\rho_{q+1}\leq 2\delta_{q+1}^{1/2}$, so that $\{\rho_q\}$ is a Cauchy sequence in $C^0(\mathcal{M})$. From $(1)_q-(3)_q$ we can also deduce that $\{h_q\}$ is a Cauchy sequence in $C^0(\mathcal{M})$; indeed, this follows from the formula  $(1)_q$, the fact that $u_q^\sharp e$ and $\rho_q^2$ are Cauchy sequences, and \eqref{e:inductive-h-j}. 

Furthermore, since $\supp\tilde\chi_q\subset\Sigma_q$ and $\bigcap_q\Sigma_q=\Sigma$, using $(2)_q$ we see that for any $x\in \mathcal{M}\setminus \Sigma$ there exists $q_0=q_0(x)$ such that 
$$
(u_{q}, \rho_{q}, h_{q})=(u_{q_0}, \rho_{q_0}, h_{q_0})
$$
for all $q\geq q_0(x)$. Similarly, since $\supp\tilde\chi_q\subset\{\rho>\delta_{q+1}^{1/2}\}$, $(u_q, \rho_q, h_q)$ agrees with $(u,\rho,h)$ on $S$. Thus there exist
\begin{align*}
\bar{u}&\in C^1(\mathcal{M})\cap C^2(\mathcal{M}\setminus \Sigma),\\
\bar{\rho}&\in C^0(\mathcal{M})\cap C^1(\mathcal{M}\setminus\Sigma),\\
\bar{h}&\in C^0(\mathcal{M}, \R^{2\times2})\cap C^1(\mathcal{M}\setminus\Sigma, \R^{2\times 2}),
\end{align*}
such that 
$$
u_q\rightarrow \bar{u}, \quad u_q^\sharp e\rightarrow \bar{u}^\sharp e, \quad \rho_q\rightarrow\bar{\rho},\quad h_q\rightarrow \bar{h}\textrm{ uniformly on }\mathcal{M}.
$$
The limit $(\bar{u}, \bar{\rho}, \bar{h})$ satisfies
$$
g-\bar{u}^\sharp e=\bar{\rho}^2(g+\bar{h})\textrm{ on }\mathcal{M}
$$
using $(1)_q$, 
$$
\|\bar{u}-u\|_0\leq \sum_{q=1}^\infty\|u_q-u_{q-1}\|_0\leq \overline{C}A^{-1}\sum_{q=1}^\infty2^{-q-1}= \frac{1}{2}\overline{C}A^{-1}\leq A^{-1/2}
$$
using $(2)_q$ and ensuring $A$ is large enough to absorb the constant $\overline{C}$,
and, using $(3)_q$,
\begin{align*}
|\nabla^2\bar{u}|\leq A^{b^2}\bar{\rho}^{1-\frac{b^2}{\theta}},\quad & |\nabla\bar{\rho}|\leq A^{b^2}\bar{\rho}^{1-\frac{b^2}{\theta}},\\
|\bar{h}|\leq A^{-\frac{\theta\alpha}{2b^2}}\bar{\rho}^{\frac{\alpha}{2b^2}},\quad &|\nabla \bar{h}|\leq A^{{b^2}-\frac{\theta\alpha}{2b^2}}\bar{\rho}^{\frac{\alpha}{2b^2}-\frac{b^2}{\theta}}.
\end{align*}
Finally, from Lemma \ref{l:rho} and \eqref{e:chi-define} we see that $\rho_q\leq 2\delta_{q+1}^{1/2}$ on $\Sigma$. Combined with the observation above that for any $x\notin \Sigma\supset S$ we have $\bar{\rho}(x)=\rho_q(x)>0$ for some $q$, we deduce $\{\bar{\rho}=0\}=\Sigma$. This proves that $(\bar{u}, \bar{\rho}, \bar{h})$ is an adapted short immersion with respect to $\Sigma\supset S$ with exponent $\theta'=\frac{\theta}{b^2}$, and satisfying \eqref{e:inductive-conclusion} as required. The proof of Proposition \ref{p:inductive} is completed.

\bigskip

\subsection{The case $n\geq 3$}\label{ss:n=3}

The key difference in the higher dimensional situation is that we have to use Corollary \ref{c:metric-2} instead of Corollary \ref{c:metric-1}. Concerning the geometric setup this involves controlling in addition the oscillation of the metric in each chart $\Omega_k$. But this merely requires choosing the finite atlas $\{\Omega_k\}$ for $\mathcal{M}$ in such a way that $\textrm{osc}_{\Omega_k}G<r_0$ for any $k$ -- this can be assumed without loss of generality. 

The difference in estimates in the two corollaries, i.e. \eqref{e:2stage-2}-\eqref{e:2stage-5} instead of \eqref{e:stage-2}-\eqref{e:stage-5}, enters in conclusions \eqref{e:v-q+1-c2}-\eqref{e:error-q+1--2}. In order to retain the same inductive estimates $(3)_q$-$(4)_q$ we therefore choose
\begin{equation}\label{e:theta-b-condition2}
b=1+\frac{2n_*\alpha\theta}{1-(2n_*+1)\theta}
\end{equation}
which in turn requires $0<\theta<\frac{1}{2n_*+1}=\frac{1}{1+n+n^2}$. With these changes Proposition \ref{p:inductive} continues to hold for the case $n\geq 3$.

\section{Proof of main Theorems}\label{s:proof-of-theorem}

We will concentrate on the case of immersions. The extension to embeddings is straight-forward and follows well-established strategies (see \cite{Nash54,SzLecturenotes,DIS15}). 

With Proposition \ref{p:inductive} at our disposal (as well as the higher dimensional analogue explained in Section \ref{ss:n=3}), the strategy for proving Theorem \ref{t:global} (as well as Theorem \ref{t:global-n}) for immersions is clear: we perform an induction on dimension on the skeleta of a given regular triangulation of $\mathcal{M}$.

 Recall that, given $(\mathcal{M},g)$ we have fixed a finite atlas of charts $\{\Omega_k\}$ on $\mathcal{M}$, and in addition we fix a triangulation $\mathcal{T}$ on $\mathcal{M}$ whose skeleta consist of a finite union of $C^1$ submanifolds, such that each triangle $T\in\mathcal{T}$ is contained in a single chart. 

By compactness of $\mathcal{M}$ and a simple mollification argument we may assume that the given short immersion $u$ is smooth and strictly short. In a first step we construct an adapted short immersion $\tilde u$ of $\mathcal{M}$ with respect to $\Sigma=\emptyset$ (i.e. $\tilde u$ is strictly short and satisfies the estimates \eqref{e:adapt-final-2}).  

\begin{proposition}\label{p:strong}
Let $u\in C^2(\mathcal{M};\R^{n+1})$ be a strictly short immersion. There exists $0<\delta^*\leq 1/8$ and $A^*\geq 1$, depending on $u$ and $g$, such that for any $A\geq A^*$ there exists a strong short immersion $\tilde{u}$ and associated $\tilde{h}$ with
\begin{equation}\label{e:strong-1}
g-\tilde{u}^\sharp e=\delta^*(g+\tilde{h}),
\end{equation}
with
\begin{equation}\label{e:strong-2}
\tfrac{1}{2}g\leq \tilde{u}^\sharp e\leq g
\end{equation}
and such that the following estimates hold:
\begin{align}
&\|\tilde{u}-u\|_0\leq \delta^*A^{-\alpha^*}, \quad \|\tilde{u}\|_2\leq A, \label{e:strong-u}\\
&\|\tilde{h}\|_0\leq A^{-\alpha^*} ,\quad \|\tilde{h}\|_1\leq A^{1-\alpha^*}. \label{e:strong-h}
\end{align}
The exponent $\alpha^*$ only depends on $\mathcal{M}$.
\end{proposition}

\begin{proof}
Since $u$ is strictly short and $\mathcal{M}$ is compact, there exists $0<\delta^*\leq 1/8$ such that
\begin{align*}
g-u^\sharp e-\delta^* g \geq \delta^* g.
\end{align*}
Using Lemma 1 in \cite{Nash54} (see also Lemma 1 in \cite{SzLecturenotes}), we decompose the metric error into finite number of primitive metrics in the different charts $\Omega_k$ as
\begin{equation}
\label{e:strong-decompose}
g-u^\sharp e-\delta^* g=\sum_{k, j}^{N_*}(\phi_k(x)a_j(x))^2\xi_j\otimes\xi_j,
\end{equation}
where $\xi_j\in \mathbb S^2$. 
Utilizing Proposition \ref{p:stage} (or Proposition 3.2 in \cite{CS19}), we obtain for any sufficiently large $\Lambda\gg 1$ an immersion $\tilde{u}$ and metric error $\mathcal{E}$ such that
\begin{align*}
\tilde{u}^\sharp e=u^\sharp e+\sum^{N_*}_{k, j}(\phi_k(x)a_j(x))^2\xi_j\otimes\xi_j+\mathcal{E}
\end{align*}
with
\begin{equation}
\label{e:strong-error}
\begin{split}
\|\tilde{u}-u\|_0\leq\frac{C_1}{\Lambda}, &\qquad \|\tilde{u}\|_2\leq C_1\Lambda^{N^*},\\
\|\mathcal{E}\|_0\leq\frac{C_1}{\Lambda}, &\qquad \|\mathcal{E}\|_1\leq C_1\Lambda^{N_*-1},
\end{split}
\end{equation}
where $C_1$ is a constant depending only on $u, g$. Thus using \eqref{e:strong-decompose} we have
\begin{align*}
g-\tilde{u}^\sharp e=\delta^* \left(g-\frac{\mathcal{E}}{\delta^*}\right),
\end{align*}
so that \eqref{e:strong-1} holds with $\tilde{h}=-\frac{\mathcal{E}}{\delta^*}$.
Now choose
$$
\Lambda=\frac{C_1}{\delta^*}A^{\alpha^*}
$$
with $\alpha^*=\tfrac{1}{8N^*}$. Then, for $A\gg 1$ sufficiently large (depending on $C_1$, $N^*$ and $\delta^*$)
we have
$$
C_1\Lambda^{N^*}=\frac{C_1^{N^*+1}}{{\delta^*}^{N^*}}A^{\alpha^* N^*}=\frac{C_1^{N^*+1}}{{\delta^*}^{N^*}}A^{1/8}\leq A
$$
and similarly $C_1\Lambda^{N^*-1}\leq A^{1-\alpha^*}$. Consequently, for $A$ sufficiently large the estimates
\eqref{e:strong-2}-\eqref{e:strong-h} hold.
\end{proof}

\bigskip

Next, fix $\theta_0<1/5$ (or $\theta_0<(1+n+n^2)^{-1}$ if $n\geq 3$) and let $0<\alpha_0<\frac{\alpha^*}{\theta_0}$. 
Set $u_0=\tilde u$, $h_0=\tilde h$ as obtained from Proposition \ref{p:strong} with $A=A_0$ sufficiently large (to be determined below), and also
$\rho_0^2=\delta^*$. From  \eqref{e:strong-u} we deduce
\begin{equation}\label{e:u-tilde-u}
\|u-u_0\|_0\leq\frac{\eps}{4}
\end{equation}
by assuming $A_0$ is sufficiently large. From \eqref{e:strong-u}-\eqref{e:strong-h}, we further have
\begin{align*}
\|\nabla^2u_0\|_0&\leq A_0\leq A_0(\delta^*)^{\frac{1}{2}-\frac{1}{2\theta_0}}\,,\\
\|h_0\|_0&\leq A_0^{-\alpha^*}\leq A_0^{-\alpha_0\theta_0}(\delta^*)^{\frac{\alpha_0}{2}}\,,\\
\|\nabla h_0\|_0&\leq A_0^{1-\alpha^*}\leq A_0^{1-\alpha_0\theta_0}(\delta^*)^{\frac{\alpha_0}{2}-\frac{1}{2\theta_0}}\,.
\end{align*}
Therefore we deduce that $u_0$ is an adapted short immersion with respect to the empty set $\Sigma_0=\emptyset$ with exponent $\theta_0$, and furthermore the estimates \eqref{e:inductive-u-rho-assumption} are satisfied by $(u_0, \rho_0, h_0)$ with $(A,\theta, \alpha)$ replaced by $(A_0,\theta_0, \alpha_0)$.

 Thus we can apply Proposition \ref{p:inductive} to obtain a $C^{1, \theta_1}$ adapted short immersion $(u_1, \rho_1, h_1)$ with respect to $\Sigma_1=\mathcal{V}$, where $\mathcal{V}$ is the vertex set of the triangulation $\mathcal{T}$ and such that \eqref{e:inductive-u-rho-assumption}\eqref{e:inductive-conclusion} hold with
$$
A_1=A_0^{b_0^2}, \quad \theta_1=\frac{\theta_0}{b_0^2}, \quad  \alpha_1=\frac{\alpha_0}{2b_0^2},
$$
where
$$
b_0=\begin{cases}1+\frac{4\alpha_0\theta_0}{1-5\theta_0}\,,& n=2;\\ 1+\frac{2n_*\alpha_0\theta_0}{1-(2n_*+1)\theta_0}\,,& n\geq 3. \end{cases}
$$
Obviously we can continue this process along the skeleta $\Sigma_1\subset\Sigma_2\subset\dots\Sigma_{n+1}=\mathcal{M}$ and obtain 
adapted short immersions $(u_j, \rho_j, h_j)$ with respect to $\Sigma_j$, $j=1,2,\dots,n+1$
with 
$$
A_{j+1}=A_j^{b_j^2}, \quad \theta_{j+1}=\frac{\theta_j}{b_j^2}, \quad  \alpha_{j+1}=\frac{\alpha_j}{2b_j^2},
$$
where
$$
b_j=\begin{cases}1+\frac{4\alpha_j\theta_j}{1-5\theta_j}\,,& n=2;\\ 
1+\frac{2n_*\alpha_j\theta_j}{1-(2n_*+1)\theta_j}\,,& n\geq 3. \end{cases}
$$
After $n+1$ steps we finally obtain a global $C^{\theta_{n+1}}$ isometric immersion $v:=u_{n+1}$ of $\mathcal{M}$, with
$$
\theta_{n+1}=(\Pi_{j=0}^nb_j^{-2})\theta_0.
$$
Note that for any fixed $\theta_0$ the recursively defined exponents $b_j$ each satisfy $b_j\to 1$ as $\alpha_0\to 0,$ and consequently also $\theta_{n+1}\to \theta_0$. Thus, for any $\theta'<\theta_0$ there exists a choice of $\alpha_0>0$ so that $\theta'<\theta_{n+1}<\theta_0$. In this way we can achieve any exponent $\theta<1/5$ for $n=2$ and $\theta<(1+n+n^2)^{-1}$ for $n\geq 3$. 
Finally, observe that (recalling \eqref{e:u-tilde-u})
\begin{align*}
\|u-v\|_0&\leq\|u-u_0\|_0+\sum_{j=0}^n\|u_{j+1}-u_j\|_0\\
&\leq \eps/4+\sum_{j=0}^nA_j^{-1/2}\leq \eps/4+(n+1)A_0^{-1/2}\\
&\leq \eps
\end{align*}
by choosing $A_0$ sufficiently large. 
This completes the proof of Theorem \ref{t:global} and Theorem \ref{t:global-n}.

%



\bigskip

\begin{thebibliography}{999}
\bibitem{ACFJK}
K. Astala, A. Clop, D. Faraco, J. J\"a\"askel\"ainen, and A. Koskia,   Nonlinear Beltrami operators, Schauder estimates and bounds for the Jacobian,
{\it Ann. Inst. H. Poincar\'e Anal. Non Lin\'eaire} {\bf 34}(2017), 6, 1543-1559.

\bibitem{Borisov58}
Yu. F. Borisov, The parallel translation on a smooth surface. {I-IV}.
{\em Vestnik Leningrad. Univ. 13,14}, (1958,1959).

\bibitem{BorisovRigidity1}
Yu. F. Borisov, On the connection between the spatial form of smooth surfaces and
  their intrinsic geometry. {\em Vestnik Leningrad. Univ. 14}, 13 (1959), 20--26.


\bibitem{Bor65}
Yu. F. Borisov, $C^{1, \alpha}$-isometric immersions of Riemannian spaces. {\it Doklady} {\bf 163}
(1965), 869-871.

\bibitem{Bor04}
Yu. F. Borisov,  Irregular $C^{1, \beta}$-surfaces with analytic metric. {\it Sib. Mat. Zh. 45,} {\bf 1}
(2004), 25-61.
%

\bibitem{BDIS15}
T. Buckmaster, , C. De Lellis, P. Isett, and L. Sz \'{e}kelyhidi, Jr.,  Anomalous
dissipation for 1/5-H\"{o}lder Euler flows. {\it Ann. of Math. (2)} {\bf 182} (2015), no. 1 , 127-172.

\bibitem{BDSV17}
T. Buckmaster, C. De Lellis, L. Sz\'ekelyhidi Jr., and V. Vicol, Onsager's conjecture for admissible weak solutions {\it Comm. Pure Appl. Math.} {\bf 72} (2019), no. 2, 229-274.

\bibitem{Buckmaster:2017wfa}
T. Buckmaster and V. Vicol, {{Nonuniqueness of weak solutions to the Navier-Stokes equation}}, {\it Ann. of Math. (2)} {\bf 189} (2019), no. 1, 101-144.


\bibitem{CS19}
W. Cao, L. Sz\'ekelyhidi Jr., $C^{1, \alpha}$ isometric extension, {\it Comm. PDE}  (2019), no. 7, 613-636.



\bibitem{CDK15}
E. Chiodaroli, C. De Lellis, and O. Kreml, Global ill-posedness of the isentropic system of gas dynamics. {\it Comm. Pure Appl. Math.} 68 (2015), No. 7, 1157-1190

\bibitem{Chern:1955jf}
S.~S.~Chern, An elementary proof of the existence of isothermal parameters on a surface, {\it Proc. Amer. Math. Soc.},
{\bf 6} {1955}, no. 5, 771-782.


\bibitem{CDS12}
 S. Conti,  C., De Lellis, , and L. Sz\'ekelyhidi Jr,  $h$-principle and rigidity for
$C^{1, \alpha}$ isometric embeddings. {\it Nonlinear partial differential equations, The Abel Symp. 2010}, Springer-Verlag, 2012, pp. 83-116

\bibitem{CohnVossen}
S. Cohn-Vossen, Zwei S\"atze \"uber die Starrheit der Eifl\"achen. {\it Nachrichten Ges.~d.~Wiss zu G\"ottingen} (1927), 125-134.

\bibitem{D14}
S. Daneri,
Cauchy problem for dissipative H\"{o}lder solutions to the incompressible Euler equations.  {\it Comm. Math. Phys.} 329 (2014), no. 2, 745-786.

\bibitem{DS17}
S. Daneri; L. Sz\'{e}kelyhidi, Jr. Non-uniqueness and h-principle for H\"{o}lder-continuous weak solutions of the Euler equations. {\it Arch. Ration. Mech. Anal.} 224 (2017), no. 2, 471-514.

\bibitem{DeLellis:2018ub}
C. De Lellis and D. Inauen, {{C$^{1,\alpha}$ Isometric Embeddings of Polar Caps}},
{\it arXiv preprint} {1809.04161v1} (2018)


\bibitem{DIS15}
C. De Lellis, D. Inauen, and L. Sz\'ekelyhidi Jr. A Nash-Kuiper theorem for $C^{1, 1/5-\delta}$ immersions on surfacesin 3 dimensions, {\it Rev. Mat. Iberoamericana} 34 (2018), 1119-1152.

\bibitem{DS09}
C. De Lellis, L. Sz\'ekelyhidi Jr., The Euler equations as a differential
inclusion. {\it Ann. of Math. (2)} 170, {\bf 3} (2009), 1417-1436.

\bibitem{DS13}
C. De Lellis, L. Sz\'ekelyhidi Jr., Dissipative continuous Euler flows. {\it Invent.
Math. } 193, {\bf 2} (2013), 377-407.

\bibitem{DS14}
C. De Lellis, L. Sz\'ekelyhidi Jr., Dissipative Euler flows and Onsager's
conjecture. {\it J. Eur. Math. Soc. (JEMS)} 16,{\bf 7} (2014), 1467-1505.

\bibitem{DeLellis:2017dt}
C. De Lellis and L. Sz{\'e}kelyhidi Jr, {{High dimensionality and h-principle in PDE}}.
{\it Bull. Amer. Math. Soc. (N.S.)} 54, {\bf 2}
(2017), 247-282.



\bibitem{GG11}
M. Ghomi and R.E. Greene, Relative isometric embeddings of Riemannian manifolds.
{\it Trans. Amer. Math. Soc.} 363 {\bf 1 }(2011), 63-73.

\bibitem{Gladbach:2019ti}
P. Gladbach and H. Olbermann, {Coarea formulae and chain rules for the Jacobian determinant in fractional Sobolev spaces}, {\it arXiv preprint} 1903.07420v1 (2019).


\bibitem{Gr73}
 M. Gromov,  Convex integration of differential relations. {\it Izv. Akad. Nauk U.S.S.R.}
{\bf 37} (1973), 329-343.

\bibitem{Gr86}
M. Gromov,  Partial Differential Relations. {\it Springer-Verlag,} 1986.

\bibitem{Gromov:2015tua}
M. Gromov, {{Geometric, algebraic, and analytic descendants of Nash isometric embedding theorems}}, {\it Bull. Amer. Math. Soc. (N.S.)} 54, {\bf 2} (2017) 173-245.


\bibitem{Herglotz:1943je} 
G. Herglotz, {\"Uber die Starrheit der Eifl\"achen}. {\it Abh. Math. Sem. Univ. Hamburg} 1 {\bf 15} (1943), 127-129.


\bibitem{HW17}
N. Hungerb\"uhler, M. Wasem. The One-Sided Isometric Extension Problem. {\it Results Math. } {\bf 71} (2017), no. 3-4, 749-781.


\bibitem{IV15}
P. Isett and V. Vicol, {H\"older continuous solutions of active scalar equations.} 
{\it Ann. PDE} {\bf 77} (2015), 1.

\bibitem{Ise16}
P. Isett, A Proof of Onsager's Conjecture.
{\it Ann. of Math. (2)} {\bf188} (2018), no. 3, 871-963.

\bibitem{Kui55}
N. Kuiper. On $C^1$ isometric embeddings I, II. {\it Proc. Kon. Acad. Wet. Amsterdam A} {\bf 58} (1955), 545-556, 683-689.

\bibitem{LP17}
M. Lewicka and M. R. Pakzad,  Convex integration for the Monge-Amp\`{e}re equation in two dimensions. {\it Anal. PDE} {\bf10} (2017), 3, 695-727

\bibitem{Nash54}
J. Nash, $C^1$ isometric imbeddings. {\it Ann. Math.} {\bf 60} (1954), 383-396.

\bibitem{PogorelovRigidity}
A. Pogorelov, The rigidity of general convex surfaces. {\em Doklady Acad. Nauk SSSR {\bf 79}\/} (1951), 739-742.

\bibitem{SzLecturenotes}
L. Sz\'ekelyhidi, Jr.  From isometric embeddings to turbulence. {\it HCDTE lecture notes. Part II. Nonlinear hyperbolic PDEs, dispersive and transport equations,} 63 pp., AIMS Ser. Appl. Math., 7, Am. Inst. Math. Sci. (AIMS), Springfield, MO, 2013.

\end{thebibliography}
\end{document}